\documentclass{article}[12pt]
\usepackage{amsmath,amssymb,amscd,amsthm}
\usepackage{graphics}
\usepackage[dvips]{graphicx}
\usepackage{latexsym}
\usepackage{amsmath,amsfonts}
\usepackage[ansinew]{inputenc}
\usepackage{dsfont}
\usepackage{enumerate}
\usepackage{cite}
\usepackage{tikz}
\usepackage{hyperref}

\theoremstyle{plain}
\newtheorem{theorem}{Theorem}[section]
\newtheorem{corollary}[theorem]{Corollary}
\newtheorem{proposition}[theorem]{Proposition}
\newtheorem{lemma}[theorem]{Lemma}

\theoremstyle{definition}
\newtheorem{definition}[theorem]{Definition}
\newtheorem{remark}[theorem]{Remark}

\DeclareMathOperator{\Char}{char}

\DeclareMathOperator{\Inn}{IAut}
\DeclareMathOperator{\Mult}{MAut}
\DeclareMathOperator{\Aut}{Aut}

\DeclareMathOperator{\Hom}{Hom}
\DeclareMathOperator{\End}{End}
\DeclareMathOperator{\Der}{Der}
\DeclareMathOperator{\ADer}{ADer}
\DeclareMathOperator{\IDer}{IDer}
\DeclareMathOperator{\Id}{id}

\begin{document}

\title{\textbf{Anti-isomorphisms and involutions on the idealization of the incidence space over the finitary incidence algebra}}
\author{\'{E}rica Z.~Fornaroli\thanks{Corresponding author: ezancanella@uem.br} \ and Roger E.~M.~Pezzott \\
\\
\textit{{\normalsize Departamento de Matem\'atica, Universidade
Estadual de Maring\'a,}} \\ \textit{{\normalsize Avenida Colombo,
5790, Maring\'a, PR, 87020-900, Brazil.}}}
\date{}

\maketitle

\begin{abstract}
Let $K$ be a field and $X$ a partially ordered set (poset). Let $FI(X,K)$ and $I(X,K)$ be the finitary incidence algebra and the incidence space of $X$ over $K$, respectively, and let $D(X,K)=FI(X,K)(+)I(X,K)$ be the idealization of the $FI(X,K)$-bimodule $I(X,K)$. In the first part of this paper, we show that $D(X,K)$ has an anti-automorphism (involution) if and only if $X$ has an anti-automorphism (involution). We also present a characterization of the anti-automorphisms and involutions on $D(X,K)$. In the second part, we obtain the classification of involutions on $D(X,K)$ to the case when $\Char K \neq 2$ and $X$ is a connected poset such that every multiplicative automorphism of $FI(X,K)$ is inner and every derivation from $FI(X,K)$ to $I(X,K)$ is inner (in particular, when $X$ has an element that is comparable with all its elements).\\

\noindent{\bf Keywords:} anti-automorphism, involution, finitary incidence algebras, idealization

\noindent{\bf 2010 MSC:} Primary: 16W10, 16W20; Secondary: 16D20
\end{abstract}

\maketitle

\section{Introduction}
Let $X$ be a partially ordered set (poset, for short) and let $K$ be a field. The \emph{incidence space} $I(X,K)$ of $X$ over $K$ is the $K$-space of functions $f:X\times X\to K$ such that $f(x,y)=0$ if $x\nleq y$. Consider the subspace $FI(X,K)$ formed by all fuctions $f\in I(X,K)$ such that for any $x,y\in X$, $x\leq y$, there is only a finite number of subintervals $[u,v]\subseteq[x,y]$ such that $u\neq v$ and $f(u,v)\neq 0$. This is a $K$-algebra under the convolution product
		$$(fg)(x,y)=\sum_{x\leq z\leq y}f(x,z)g(z,y),$$
for any $f, g\in FI(X,K)$, called the \emph{finitary incidence algebra} of $X$ over $K$. Moreover, $I(X,K)$ is an $FI(X,K)$-bimodule~\cite{KN}. When every interval of $X$ is finite, $X$ is called locally finite and, in this case, $FI(X,K)=I(X,K)$ is known as the \emph{incidence algebra} of $X$ over $K$~\cite{Rota}. When $X$ is a chain of cardinality $n$, then $FI(X,K)\cong UT_n(K)$, the algebra of upper triangular matrices over $K$.

Involutions on $FI(X,K)$ were first studied by Scharlau~\cite{Scharlau} and many years later by Spiegel~\cite{Spiegel05, Spiegel08}. The results \cite[Theorem~2.1(b)]{Scharlau} and \cite[Theorem~1]{Spiegel08} about the classification of involutions on $FI(X,K)$ are incorrect as indicated in~\cite{BL}. In \cite{DKL} was given a complete classification of involutions on $UT_n(K)$ for the case when $\Char K\neq 2$. Later, Brusamarello, Fornaroli and Santulo considered the case when $\Char K\neq 2$ and $X$ has an element that is comparable with all its elements and obtained a general classification of  involutions on $FI(X,K)$ for the case when $X$ is finite in \cite{BFS11} and locally finite in \cite{BFS12}. Finally, in \cite{BFS14}, they generalized this classification for the case when $X$ is connected and every multiplicative automorphism of $FI(X,K)$ is inner.

Since $I(X,K)$ is an $FI(X,K)$-bimodule, it is natural to consider the idealization $D(X,K)=FI(X,K)(+)I(X,K)$ of $I(X,K)$. Such idealization was first studied by Dugas~\cite{DMH} and then by Dugas and Wagner~\cite{DW}. In this paper we prove that $D(X,K)$ has an involution if and only if $X$ has an involution (Corollary~\ref{coroiffinvolution}). We also give a characterization of the anti-automorphisms and involutions on $D(X,K)$ (Theorem~\ref{teoDecinvDP}) similar to the one obtained in \cite{DW} for the automorphisms of $D(X,K)$. Finally, we consider a connected poset $X$ such that every multiplicative automorphism of $FI(X,K)$ is inner and every derivation from $FI(X,K)$ to $I(X,K)$ is inner, and we give a general classification of involutions on $D(X,K)$ when $\Char K\neq 2$.

The descriptions of the automorphisms and derivations of $FI(X,K)$ and their generalizations were given in \cite{AutK,KD}. Conditions for a multiplicative automorphism of $FI(X,K)$ to be inner and for a derivation of $FI(X,K)$ to be inner were given in \cite{BMUL,ER} for the case when $X$ is finite and connected.

Our work is organized as follows. In Section~\ref{SPreli} we fix the notations and recall some results that will be used throughout the work. In Section~\ref{SantieInvo} we consider a field $K$ and posets $X,Y$ and prove that there is an anti-isomorphism $D(X,K) \to D(Y,K)$ if and only if there is an anti-isomorphism $X \to Y$ (Theorem~\ref{teoantiisoDP}). As a consequence, $D(X,K)$ has an involution if and only if $X$ has an involution (Corollary~\ref{coroiffinvolution}). In Section~\ref{SAutandDer} we give a characterization of the automorphisms of $D(X,K)$ in the case when $X$ is connected, every multiplicative automorphism of $FI(X,K)$ is inner and every derivation from $FI(X,K)$ to $I(X,K)$ is inner (Proposition~\ref{propallcomparableAutDP}). In Section~\ref{SCharac} we present a characterization of the anti-automorphisms and involutions on $D(X,K)$ (Theorem~\ref{teoDecinvDP}) and also a more simplified characterization if $X$ is a connected poset such that every multiplicative automorphism of $FI(X,K)$ is inner and every derivation from $FI(X,K)$ to $I(X,K)$ is inner (Corollary~\ref{coroallcomparableAntiauttDP}). In Section~\ref{Sclassifi}, we consider a field $K$ of characteristic different from $2$ and a connected poset $X$ such that every multiplicative automorphism of $FI(X,K)$ is inner and every derivation from $FI(X,K)$ to $I(X,K)$ is inner and we give the classification of involutions on $D(X,K)$ via inner automorphisms (Theorems~\ref{teoclassivazio} and \ref{teoclassinoempty}). Finally, we use such a classification to obtain the general classification of involutions on $D(X,K)$ (Theorems~\ref{teoeqivaGeral}, \ref{teoclassigeralnoFP}, \ref{teoclassigeralonepoint} and \ref{teoclassigeralfinal}).

\section{Preliminaries} \label{SPreli}

\subsection{Rings and modules}\label{SPreli1}

Throughout the paper rings are associative with unity. The center of a ring $R$ is denoted by $Z(R)$ and the set of invertible elements of $R$ is denoted by $U(R)$.

Let $R$ be a ring and let $M, N$ be left $R$-modules. We use the notations $\Hom_R(M,N) = \{ \phi: M \to N : \phi$  is  $R$-linear$\}$, $\End_R(M)=\Hom_R(M,M)$ and $\Aut_R(M)=U(\End_R(M))$. Each $\Phi \in  \End_R(M\oplus N)$ has a presentation
$$ \Phi= \begin{bmatrix}
\phi_{11}& \phi_{12}\\
\phi_{21}&  \phi_{22}
\end{bmatrix},$$
where $\phi_{11} \in \End_R(M)$, $\phi_{12} \in \Hom_R(N,M)$, $\phi_{21} \in \Hom_R(M,N)$ and $\phi_{22} \in \End_R(N)$, such that for any $\begin{bmatrix}
m\\n
\end{bmatrix} \in M\oplus N$,
$$ \Phi\left(\begin{bmatrix}
m\\n
\end{bmatrix} \right) = \begin{bmatrix}
\phi_{11}(m) + \phi_{12}(n)\\ \phi_{21}(m) + \phi_{22}(n)
\end{bmatrix} .$$

Let $R$ be a ring and $M$ an $R$-bimodule. We recall that  $R\oplus M$ is a ring with the multiplication
$$\begin{bmatrix}r\\m \end{bmatrix}\begin{bmatrix}s\\n \end{bmatrix} = \begin{bmatrix}rs\\rn +ms \end{bmatrix},$$
for all $r,s \in R$ and $m,n \in M$. This ring is denoted by $R(+)M$ and is called \emph{idealization of the $R$-bimodule $M$} (or \emph{trivial extension of $R$ by $M$}). Observe that $0_{R(+)M} = \begin{bmatrix}0_R\\0_M \end{bmatrix}$ and $1_{R(+)M} = \begin{bmatrix}1_R\\0_M \end{bmatrix}.$ It is easy to check that $\begin{bmatrix}r\\m \end{bmatrix} \in U(R(+)M)$ if and only if $r \in U(R)$ and, in this case, $\begin{bmatrix}r\\m \end{bmatrix}^{-1} = \begin{bmatrix}r^{-1}\\-r^{-1}mr^{-1} \end{bmatrix}$.

For any $R$-bimodule $M$, we use the notations $C_{M}(R) = \{ m \in M: mr=rm,$ for all $r \in R\}$ and $C_{R}(M) = \{ r \in R: rm = mr,\text{ for all } m \in M\}$.

\begin{proposition} \label{propcentroidealizacaoRmodulos}
Let $R$ be a ring and $M$ an $R$-bimodule. Then
$$Z(R(+)M)= \left\{\begin{bmatrix}r\\m \end{bmatrix}: r \in Z(R)\cap C_{R}(M) \text{ and } m \in C_{M}(R)  \right\}.$$
\end{proposition}
\begin{proof}
Let $s \in R$ and $n \in M$. If $\begin{bmatrix}r\\m \end{bmatrix} \in Z(R(+)M)$, then
$\begin{bmatrix}r\\m \end{bmatrix} \begin{bmatrix}s\\0 \end{bmatrix} = \begin{bmatrix}s\\0 \end{bmatrix}\begin{bmatrix}r\\m \end{bmatrix}$ and $\begin{bmatrix}r\\m \end{bmatrix} \begin{bmatrix}0\\n \end{bmatrix} = \begin{bmatrix}0\\n \end{bmatrix}\begin{bmatrix}r\\m \end{bmatrix}.$
Thus, $r \in Z(R)\cap C_{R}(M)$ and $m \in C_{M}(R)$. On the other hand, let $\begin{bmatrix}r\\m \end{bmatrix} \in R(+)M$ with $r \in Z(R)\cap C_{R}(M)$  and $m \in C_{M}(R)$. Then
$$\begin{bmatrix}r\\m \end{bmatrix} \begin{bmatrix}s\\n \end{bmatrix} = \begin{bmatrix}rs\\rn+ms \end{bmatrix}= \begin{bmatrix}sr\\nr+sm \end{bmatrix} =  \begin{bmatrix}s\\n \end{bmatrix}\begin{bmatrix}r\\m \end{bmatrix}$$
and, therefore, $\begin{bmatrix}r\\m \end{bmatrix} \in Z(R(+)M)$.
\end{proof}

\subsection{Automorphisms and involutions of an algebra} \label{SPreli2}

Let $K$ be a field and let $A$ be a $K$-algebra. We denote by $\Aut(A)$ the group of ($K$-linear) automorphisms of $A$. For any $a \in U(A)$, $\Psi_a$ denotes the inner automorphism defined by $a$, i.e., $\Psi_a : A\to A$ is such that $\Psi_a(x) = axa^{-1}$ for all $x \in A$. The group of inner automorphisms of $A$ is denoted by $\Inn(A)$.

For us, involutions on $A$ are ($K$-linear) anti-automorphisms of order $2$.

\begin{proposition} \label{propAutInner}
Let $K$ be a field and $A$ a $K$-algebra. For any $a, b \in U(A)$ and $\phi, \varphi$ involutions on $A$ we have:
\begin{enumerate}
\item[(i)] $\Psi_a = \Psi_b$ if and only if $a b^{-1} \in Z(A)$.
\item[(ii)] The anti-automorphism $\Psi_a \circ \phi$ is an involution on $A$ if and only if $\Psi_{a}= \Psi_{\phi(a)}$.
\item[(iii)] If $\phi(a) = \pm a$, then $\Psi_a \circ \phi$ is an involution on $A$. On the other hand, if $A$ is a central algebra and $\Psi_a \circ \phi$ is an involution on $A$, then $\phi(a)=\pm a$.
\item[(iv)] There exists $\Psi \in \Inn(A)$ such that $\Psi\circ \phi = \varphi\circ\Psi$ if and only if $\phi = \Psi_{v\varphi(v)}\circ \varphi$, for some $v \in U(A)$.
\end{enumerate}
\end{proposition}
\begin{proof}
The proof of (i) is obvious. For (ii) just see that $(\Psi_a \circ \phi)^2 = \Psi_{a} \circ \Psi_{\phi(a)^{-1}}$. Since (iv) is \cite[Lemma~5]{BFS11}, it remains to prove (iii). If $\phi(a) = \pm a$, then $\Psi_a \circ \phi$ is an involution by (ii).  Conversely, if $A$ is a central algebra and $\Psi_a \circ \phi$ is an involution, then there is $k \in K$ such that $\phi(a)=ka$, by (i) and (ii). Thus,
\begin{align*}
    a = \phi^{2}(a) = k\phi(a) = k^2a \Rightarrow (k^2-1)a =0 \Rightarrow k^2-1=0 \Rightarrow k= \pm 1.
\end{align*}
Therefore, $\phi(a) = \pm a$.
\end{proof}
Let $\phi$ be an involution on a $K$-algebra $A$ and let $a \in A$. If $\phi(a)=a$, then $a$ is said \emph{$\phi$-symmetric} and if $\phi(a)=-a$, then $a$ is said \emph{$\phi$-skew-symmetric}.

\subsection{Finitary incidence algebras}\label{SPreli3}

Let $X$ and $Y$ be posets. An \emph{isomorphism} (resp.~\emph{anti-isomorphism}) from $X$ to $Y$ is a bijective map $\mu:X \to Y$ that satisfies the following property for any $x,y \in X$:
$$ x\leq y \Leftrightarrow \mu(x) \leq \mu(y) \text{ (resp. } \mu(y) \leq \mu(x)\text{)}.$$
When $X=Y$, $\mu$ is also called an \emph{automorphism} (resp.~\emph{anti-automorphism}). An anti-automorphism $\mu:X \to X$ of order $2$ is an \emph{involution} on $X$.

Let $K$ be a field. If $\alpha$ is an isomorphism (resp.~$\lambda$ is an anti-isomorphism) from $X$ to $Y$, then $\alpha$  (resp.~$\lambda$) induces an isomorphism (resp.~anti-isomorphism)  $\widehat{\alpha}$ (resp.~$\rho_\lambda$) from $FI(X,K)$ to $FI(Y,K)$ defined by $\widehat{\alpha}(f)(x, y) = f(\alpha^{-1}(x), \alpha^{-1}(y))$ $(\rho_\lambda(f)(x, y) = f(\lambda^{-1}(y), \lambda^{-1}(x)))$, for all $f \in FI(X,K)$ and
$x, y \in Y$. In particular, when $X =Y$,
$\widehat{\alpha}$ is an automorphism and $\rho_\lambda$ is an anti-automorphism of $FI(X,K)$. Moreover, if $\lambda$ is an involution, so is $\rho_\lambda$.

We denote by $\delta$ the unity of $FI(X,K)$ which is given by
$$\delta(x,y) =
\begin{cases}
1 & \text{if } x=y\\
0 & \text{if } x\neq y
\end{cases}.$$

Given $x,y \in X$ such that $x\leq y$, we denote by $e_{xy}$ the element of $FI(X,K)$ defined by
$$e_{xy}(u,v) =
\begin{cases}
1 & \text{if }(u,v)=(x,y)\\
0 & \text{otherwise}
\end{cases},$$
and we write $e_x$ for $e_{xx}$. If $\alpha$ is an automorphism and $\lambda$ an anti-automorphism of $X$, then
\begin{align*}
    \widehat{\alpha}(e_{xy}) = e_{\alpha(x)\alpha(y)}\quad \text{and}\quad  \rho_\lambda(e_{xy}) = e_{\lambda(y)\lambda(x)},
\end{align*}
for all $x \leq y$ in $X$.

An element $\sigma \in I(X,K)$ such that $\sigma(x,y)\neq 0$ for all $x \leq y$, and  $\sigma(x,y)\sigma(y,z)=\sigma(x,z)$  whenever $x\leq y\leq z$, determines an automorphism $M_\sigma$ of $FI(X,K)$ by $M_{\sigma}(f)(x,y)=\sigma(x,y)f(x,y)$, for all $f \in FI(X,K)$ and $x,y \in X$. Such automorphisms are called \emph{multiplicative}. We denote the set of all multiplicative automorphisms of $FI(X,K)$ by $\Mult(FI(X,K))$.

\begin{theorem} \label{DecoAutFIP}
 Let $X$ be a poset and let $K$ be a field. If $\Phi$ is an automorphism (anti-automorphism, involution) of $FI(X,K)$, then $\Phi = \Psi \circ M \circ \varphi$, where $\Psi \in \Inn(FI(X,K))$, $M\in \Mult(FI(X,K))$ and $\varphi$ is the automorphism (anti-automorphism, involution) induced by an automorphism (anti-automorphism, involution) of $X$.
\end{theorem}

For automorphisms, the theorem above follows from \cite[Lemma~3]{AutK}. For anti-automorphisms and involutions it is \cite[Theorem~3.5]{BFS14}.

\begin{remark}
The automorphism (anti-automorphism, involution) $\varphi$ in the theorem above is equal to $\widehat{\alpha}$ (resp.~$\rho_\lambda$), where $\alpha$ (resp.~$\lambda$) is the automorphism (anti-automorphism, involution) of $X$ induced by $\Phi$.
(See proofs of \cite[Lemma~3]{AutK} and \cite[Theorem~3.5]{BFS14}).
\end{remark}

Now, by \cite[Proposition~2]{DW},
\begin{equation} \label{eqcenralizadores}
    C_{FI(X,K)}(I(X,K))= C_{I(X,K)}(FI(X,K)) = Z(FI(X,K))
\end{equation}
and $Z(FI(X,K))$ is the set of all diagonal functions which are constant on each connected component of $X$. Thus, if $X$ is connected, then $FI(X,K)$ is a central algebra.

Consider the idealization $D(X,K)=FI(X,K)(+)I(X,K)$. It follows from~\eqref{eqcenralizadores} and Proposition~\ref{propcentroidealizacaoRmodulos} that
\begin{equation}\label{eqcentroDXK}
    Z(D(X,K))= \left\{\begin{bmatrix}f\\i \end{bmatrix}: f,i \in Z(FI(X,K))  \right\}.
\end{equation}
In particular, for $k_1,k_2 \in K$, $\begin{bmatrix}k_1\delta \\ k_2\delta \end{bmatrix} \in Z(D(X,K))$. To make notations shorter we will write $\begin{bmatrix}k_1\delta \\k_2\delta \end{bmatrix}=c_{k_1,k_2}$.

\begin{proposition}\label{propCenDXKconnected}
Let $K$ be a field and let $X$ be a poset. If $X$ is connected, then $Z(D(X,K)) = \{c_{k_1,k_2}: k_1,k_2 \in K\}$.
\end{proposition}
\begin{proof}
If $X$ is connected, then $FI(X,K)$ is a central algebra and the result follows from \eqref{eqcentroDXK}.
\end{proof}

\section{Anti-isomorphisms and involutions on $D(X,K)$}\label{SantieInvo}

Let $X$ and $Y$ be posets and let $K$ be a field. In \cite[Theorem~2.3]{DW}, the authors proved that $D(X,K)$ and $D(Y,K)$ are isomorphic if and only if $X$ and $Y$ are isomorphic. In this section, we prove that the same is true for anti-isomorphisms (Theorem~\ref{teoantiisoDP}). Consequently, $D(X,K)$ has an involution if and only if $X$ has an involution (Corollary~\ref{coroiffinvolution}).

\begin{definition}\label{defiSL}
Let $M$ be an $R$-bimodule and let $N$ be an $S$-bimodule. An additive map $\mathcal V: M \to N$ is $\beta$-\emph{semilinear} ($\beta$-\emph{anti-semilinear}) if there is a homomorphism (anti-homomorphism) of rings $\beta:R\to S$ such that $\mathcal V(rms) = \beta(r)\mathcal V(m)\beta(s)$ ($\mathcal V(rms) = \beta(s)\mathcal V(m)\beta(r)$),  for all $m \in M$ and $r,s \in R$.
 For simplicity we just write $\beta$-SL and $\beta$-ASL for $\beta$-semilinear and $\beta$-anti-semilinear, respectively.
\end{definition}

The next properties follow directly  from Definition~\ref{defiSL}.

\begin{proposition} \label{propcomposicao}
Let $M$ be an $R$-bimodule, $N$ an $S$-bimodule and $L$ a $T$-bimodule.
\begin{enumerate}
\item[(i)] If $\mathcal V:M\to N$ is $\beta$-SL ($\beta$-ASL) and $\mathcal W:N\to L$ is $\gamma$-SL ($\gamma$-ASL), then $\mathcal W \circ \mathcal V:M \to L$ is $ \gamma \circ \beta$-SL.
\item[(ii)] If $\mathcal V:M\to N$ is $\beta$-ASL ($\beta$-SL) and $\mathcal W:N\to L$ is $\gamma$-SL ($\gamma$-ASL), then $\mathcal W \circ \mathcal V:M \to L$ is $ \gamma \circ \beta$-ASL.
\end{enumerate}
\end{proposition}

\begin{lemma} \label{exten}
Let $\lambda: X \to Y$ be an anti-isomorphism of posets. Then $\rho_\lambda:FI(X,K) \to FI(Y,K)$ extends to a bijective  $\rho_\lambda$-ASL map $\overline{\rho_\lambda}:I(X,K) \to I(Y,K)$.
\end{lemma}
\begin{proof}
Clearly,  $\overline{\rho_\lambda} : I(X,K) \to I(Y,K)$ defined by $\overline{\rho_\lambda}(i)(x,y) = i(\lambda^{-1}(y),\lambda^{-1}(x))$, for any $i \in I(X,K)$ and  $x,y \in Y$, is $\rho_\lambda$-ASL and  $\overline{\rho_\lambda}|_{FI(X,K)}=\rho_\lambda$. Analogously, $\overline{\rho_{\lambda^{-1}}}:I(Y,K) \to I(X,K)$ given by $\overline{\rho_{\lambda^{-1}}}(j)(u,v) = j(\lambda(v),\lambda(u))$,  for all $j \in I(Y,K)$ and $u,v \in X$, is an extension $\rho_{\lambda^{-1}}$-ASL of $\rho_{\lambda^{-1}}$ such that $\overline{\rho_{\lambda^{-1}}} \circ \overline{\rho_\lambda} = \Id_{I(X,K)}$ and $\overline{\rho_\lambda}\circ \overline{\rho_{\lambda^{-1}}} = \Id_{I(Y,K)}$. Therefore, $\overline{\rho_{\lambda^{-1}}}=(\overline{\rho_\lambda})^{-1}$ and $\overline{\rho_\lambda}$ is bijective.
\end{proof}
Let $X$ and $Y$ be arbitrary posets. In the theorem below we denote by $e_{xy}^X$ and $e_{zw}^Y$ the elements $e_{xy}$ in $FI(X,K)$ and $e_{zw}$ in $FI(Y,K)$, respectively.

\begin{theorem} \label{teoantiisoDP}
Let $K$ be a field and let $X$ and $Y$ be posets. There is an anti-isomorphism $D(X,K) \to D(Y,K)$ if and only if there is an anti-isomorphism $X \to Y$.
\end{theorem}
\begin{proof}
Suppose there is an anti-isomorphism $\Phi : D(X,K) \to D(Y,K)$. Given $x \in X$, by \cite[Lemma~1]{KN}, $e_x^X$ is a primitive idempotent of $FI(X,K)$ and, consequently, $\begin{bmatrix}e_x^X\\0 \end{bmatrix}$ is a primitive idempotent of $D(X,K)$, by (1) and (2) of \cite[Proposition~5]{DW}. By (3) of \cite[Proposition~5]{DW}, there is a unique $y \in Y$ such that $\Phi\left(\begin{bmatrix}e_x^X\\0 \end{bmatrix}\right)$ is conjugate to $\begin{bmatrix}e_y^Y\\0 \end{bmatrix}$ in $D(Y,K)$. Thus, $\Phi$ induces a map $\lambda_1: X \to Y$ such that, given $x \in X$, $\Phi\left(\begin{bmatrix}e_x^X\\0 \end{bmatrix}\right)$ is conjugate to $\begin{bmatrix}e_{\lambda_1(x)}^Y\\0 \end{bmatrix}$. Similarly, $\Phi^{-1}$ induces a map $\lambda_2:Y \to X$ such that, given $z \in Y$, $\Phi^{-1}\left(\begin{bmatrix}e_z^Y\\0 \end{bmatrix}\right)$ is conjugate to $\begin{bmatrix}e_{\lambda_2(z)}^X\\0 \end{bmatrix}$. Now, we show that $\lambda_2=\lambda_1^{-1}$. Let $x \in X$, $\theta \in U(D(Y,K))$ and $\kappa \in U(D(X,K))$ such that
\begin{center}
$\Phi\left(\begin{bmatrix}e_x^X\\0 \end{bmatrix}\right) = \theta \begin{bmatrix}e_{\lambda_1(x)}^Y\\0 \end{bmatrix} \theta^{-1}$ \ and \  $\Phi^{-1}\left(\begin{bmatrix}e_{\lambda_1(x)}^Y\\0 \end{bmatrix}\right) = \kappa \begin{bmatrix}e_{\lambda_2(\lambda_1(x))}^X\\0 \end{bmatrix} \kappa^{-1}.$
\end{center}
We have
\begin{align*}
\begin{bmatrix} e_x^X\\0 \end{bmatrix}
& = \Phi^{-1}(\theta^{-1}) \Phi^{-1}\left(\begin{bmatrix}e_{\lambda_1(x)}^Y\\0 \end{bmatrix}\right)\Phi^{-1}(\theta)\\
& = \Phi^{-1}(\theta^{-1})\kappa \begin{bmatrix}e_{\lambda_2(\lambda_1(x))}^X\\0 \end{bmatrix} \kappa^{-1}\Phi^{-1}(\theta)\\
& =  \Phi^{-1}(\theta^{-1})\kappa \begin{bmatrix}e_{\lambda_2(\lambda_1(x))}^X\\0 \end{bmatrix} (\Phi^{-1}(\theta^{-1})\kappa)^{-1}.
\end{align*}
Again by (3) of \cite[Proposition~5]{DW}, $(\lambda_2 \circ \lambda_1)(x) = x$ and, consequently, $\lambda_2 \circ \lambda_1= \Id_{X}$.
Analogously, $\lambda_1 \circ \lambda_2= \Id_{Y}$ and, therefore, $\lambda_2=\lambda_1^{-1}$. To conclude that $\lambda_1$ is an anti-isomorphism, we use an argument similar to the one used in \cite[Theorem~5]{KN}, since we have
\begin{enumerate}
\item[(i)] $x \leq y \in X$ if and only if $\begin{bmatrix}e_x^X\\0 \end{bmatrix}D(X,K)\begin{bmatrix}e_y^X\\0 \end{bmatrix}\not= \left\{\begin{bmatrix}0\\0 \end{bmatrix}\right\};$
\item[(ii)] $u \leq v \in Y$ if and only if $\begin{bmatrix}e_u^Y\\0 \end{bmatrix}D(Y,K)\begin{bmatrix}e_v^Y\\0 \end{bmatrix}\not= \left\{\begin{bmatrix}0\\0 \end{bmatrix}\right\}.$
\end{enumerate}

Conversely, suppose there is an anti-isomorphism $\lambda : X \to Y$. By Lemma~\ref{exten},  $\overline{\rho_\lambda}:I(X,K) \to I(Y,K)$ is a bijective $\rho_\lambda$-ASL map such that $\overline{\rho_\lambda}|_{FI(X,K)} = \rho_\lambda$. Let
$$\begin{array}{cccc}
\Upsilon : &D(X,K) &\to & D(Y,K) \\
&\begin{bmatrix}f\\i \end{bmatrix} &\mapsto &\begin{bmatrix}\rho_\lambda(f)\\\overline{\rho_\lambda}(i) \end{bmatrix}
\end{array}.$$
It is easy to prove that $\Upsilon$ is an anti-isomorphism whose inverse is given by $\Upsilon^{-1} \left(\begin{bmatrix}h\\l \end{bmatrix}\right) = \begin{bmatrix}\rho_{\lambda^{-1}}( f)\\\overline{\rho_{\lambda^{-1}}}(i) \end{bmatrix}$, for all $\begin{bmatrix}h\\l \end{bmatrix} \in D(Y,K)$.
\end{proof}

Suppose $X = Y$ in the last theorem. If $\Phi$ is an involution, then $\Phi^{-1} = \Phi$ and, therefore, $\lambda_1^{-1}= \lambda_1$. Thus, $\lambda_1$ is an involution on $X$. On the other hand, if $\lambda$ is an involution, then $\Upsilon^{-1} = \Upsilon$ and, therefore $\Upsilon$ is an involution on $D(X,K)$. Thus, we have the following result.

\begin{corollary} \label{coroiffinvolution}
Let $K$ be a field and let $X$ be a poset. There is an involution on $D(X,K)$ if and only if there is an involution on $X$.
\end{corollary}

In \cite[Proposition~4]{DW}, the authors proved that each automorphism $\eta$ of $FI(X,K)$ uniquely extends to a bijective $\eta$-SL map $\overline{\eta}: I(X,K) \to I(X,K)$. Our Proposition~\ref{propesten} provides a similar result for anti-automorphisms.

\begin{lemma}\label{extId}
Let $\mathcal V: I(X,K) \to I(X,K)$ be a bijective $\Id_{FI(X,K)}$-SL map. If $\mathcal V|_{FI(X,K)}=\Id_{FI(X,K)}$, then $\mathcal V = \Id_{I(X,K)}$.
\end{lemma}
\begin{proof}
Let $i \in I(X,K)$ and let $x\leq y$ in $X$. Since $e_x, e_y, e_{xy} \in FI(X,K)$ and $\mathcal V$ is $K$-linear, we have
\begin{align*}
\mathcal V(i)(x,y)e_{xy}&= e_x\mathcal V(i)e_y = \mathcal V(e_xie_y) =\mathcal V(i(x,y)e_{xy}) =i(x,y)\mathcal V(e_{xy})
=i(x,y)e_{xy}.
\end{align*}
Thus, $\mathcal V(i)(x,y) = i(x,y)$ for all $x\leq y \in X$ and, therefore, $\mathcal V(i)= i$.
\end{proof}

\begin{proposition}\label{propesten}
Let $\phi$ be an anti-automorphism of $FI(X,K)$. Then $\phi$ extends uniquely to a bijective $\phi$-ASL map $\overline{\phi}:I(X,K)\to I(X,K)$. Moreover, $(\overline{\phi})^{-1} = \overline{\phi^{-1}}$.
\end{proposition}
\begin{proof}
By Theorem~\ref{DecoAutFIP}, $\phi = \Psi \circ M \circ \rho_\lambda$, where $\Psi \in \Inn(FI(X,K))$, $M \in \Mult(FI(X,K))$ and $\rho_\lambda$ is the anti-automorphism  induced by an anti-automorphism $\lambda$ of $X$. Let $\overline{\Psi}$ and $\overline{M}$ be the extensions of $\Psi$ and  $M$,  respectively, given by \cite[Proposition~4]{DW}, and let $\overline{\rho_\lambda}$ be the extension of $\rho_\lambda$ given by Lemma~\ref{exten}. Then $\overline{\phi} := \overline{\Psi}\circ \overline{M} \circ \overline{\rho_\lambda}: I(X,K) \to I(X,K)$ is a bijection that extends $\phi$, moreover,  $\overline{\phi}$ is $\phi$-ASL, by Proposition~\ref{propcomposicao}.

Now, suppose $\mathcal W,\mathcal S: I(X,K) \to I(X,K)$ are bijective $\phi$-ASL maps that extends $\phi$ and consider $\mathcal V =\mathcal S^{-1}\circ \mathcal W$. Since $\mathcal S|_{FI(X,K)} = \phi$, we have $\mathcal S^{-1}|_{FI(X,K)} = \phi^{-1}$ and, consequently, $\mathcal V|_{FI(X,K)}= \Id_{FI(X,K)}$. Moreover, $\mathcal S^{-1}$ is $\phi^{-1}$-ASL. By Proposition~\ref{propcomposicao}, $\mathcal S^{-1}\circ \mathcal W$ is $(\phi^{-1}\circ \phi)$-SL, that is, $\mathcal V$ is $\Id_{FI(X,K)}$-SL. By Lemma~\ref{extId}, $\mathcal V = \Id_{I(X,K)}$ and, therefore, $\mathcal W = \mathcal S$.

Finally, since $\overline{\phi}|_{FI(X,K)}= \phi$ and $\overline{\phi}$ is $\phi$-ASL, we have $(\overline{\phi})^{-1}|_{FI(X,K)}= \phi^{-1}$ and $(\overline{\phi})^{-1}$ is $\phi^{-1}$-ASL. Therefore, $\overline{\phi^{-1}}= (\overline{\phi})^{-1}$.
\end{proof}

\begin{corollary} \label{coroexteninvo}
If $\rho$ is an involution on $FI(X,K)$, then $(\overline{\rho})^2 = \Id_{I(X,K)}$.
\end{corollary}

\begin{remark}\label{obsCompoExten}
As consequence of the last proposition and \cite[Proposition~4]{DW}, $\overline{\phi_1 \circ \phi_2} = \overline{\phi_1}\circ \overline{\phi_2}$ for all $ \phi_1, \phi_2 \in \{\phi:FI(X,K) \to FI(X,K)$: $\phi$ is an automorphism or an anti-automorphism$\}$.
\end{remark}

\section{Automorphisms of $D(X,K)$ and derivations from $FI(X,K)$ to $I(X,K)$}\label{SAutandDer}

We recall that a $K$-linear map  $D:FI(X,K) \to I(X,K)$ is a \emph{derivation} if $D(fg) = D(f)g +fD(g)$ for all $f,g \in FI(X,K)$. The set $\Der(X,K)$ of all such derivations is a $K$-module, and for any $g \in Z(FI(X,K))$ and $D \in \Der(X,K)$, $gD \in \Der(X,K)$.

An element $\tau\in I(X,K)$ such that  $\tau(x,y)+\tau(y,z)=\tau(x,z)$ whenever $x \leq y \leq z$, determines a derivation $L_{\tau}:FI(X,K)\to I(X,K)$ by $L_{\tau}(f)(x,y)=\tau(x,y)f(x,y)$, for all $f \in FI(X,K)$ and $x,y \in X$. Such derivations are called \emph{additive}.

For any $i \in I(X,K)$, the map $D_i:FI(X,K)\to I(X,K)$ given by $D_i(f) = fi-if$ for all $f \in FI(X,K)$, is a derivation called \emph{inner derivation}. Let $\ADer(X,K)$ and $\IDer(X,K)$ be the $K$-submodules of $\Der(X,K)$ formed by the additive derivations and the inner derivations, respectively. Then
\begin{equation} \label{eqDerDec}
     \Der(X,K) = \IDer(X,K) + \ADer(X,K),
\end{equation}
by \cite[Theorem~3.2]{DW}.

\begin{proposition}\label{all-comparable}
If $X$ has an element $x_0$ that is comparable with all its elements, then $\ADer(X,K)\subseteq \IDer(X,K)$ and, therefore, $\Der(X,K) = \IDer(X,K)$.
\end{proposition}
\begin{proof}
Let $L_\tau \in \ADer(X,K)$ and define
$$f(x,y) = \begin{cases}
-\tau(x,x_0) &\text{if } x = y\leq x_0\\
\tau(x_0,x) &\text{if } x_0 \leq x= y\\
0 &\text{if } x\neq y
\end{cases}.$$
Obviously, $f \in FI(X,K)$. Let $g \in FI(X,K)$ and $x \leq y$ in $X$. Then
\begin{align*}
D_{f}(g)(x,y) = (gf)(x,y)-(fg)(x,y) = g(x,y)f(y,y)- f(x,x)g(x,y).
\end{align*}
Since $x_0$ is comparable with $x$ and $y$, we have three cases to analyse.
If $x \leq y \leq x_0$, then $\tau(x,x_0)= \tau(x,y) + \tau(y,x_0)$ and so
\begin{align*}
D_{f}(g)(x,y) & = -g(x,y)\tau(y,x_0)+\tau(x,x_0) g(x,y)\\
              & =  [\tau(x,x_0)-\tau(y,x_0)]g(x,y) \\
              & = \tau(x,y)g(x,y).
\end{align*}
The case $x_0 \leq x \leq y$ is similar and we also have $D_{f}(g)(x,y)=\tau(x,y)g(x,y)$. Finally, if $x \leq x_0 \leq y$, then
\begin{align*}
D_{f}(g)(x,y) & =  g(x,y)\tau(x_0,y)+ \tau(x,x_0)g(x,y)\\
              & =  [\tau(x,x_0)+\tau(x_0,y))]g(x,y) \\
              & = \tau(x,y)g(x,y).
\end{align*}
Thus, $L_\tau = D_{f}$ and so $\ADer(X,K) \subseteq \IDer(X,K)$. Therefore, $\Der(X,K) = \IDer(X,K)$, by \eqref{eqDerDec}.
\end{proof}

The existence of an element that is comparable with all the elements of $X$ is not necessary to ensure that every additive derivation is inner. Consider the followings posets  $X_1$ and $X_2$.
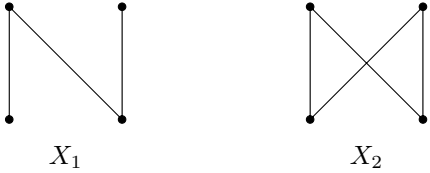
\begin{center}
\begin{tikzpicture}
\draw [fill=black] 
(-1.25,-0.5) node{$X_1$} (-2,0) circle(0.05)--(-2,.75)--(-2,1.5)circle(0.05)--(-1.25,.75)--(-.5,0)circle(0.05)--(-.5,.75)--(-.5,1.5)circle(0.05)
(2.75,-0.5) node{$X_2$} (2,0) circle(0.05)--(2,.75)--(2,1.5)circle(0.05)--(2.75,.75)--(3.5,0)circle(0.05)--(3.5,.75)--(3.5,1.5)circle(0.05)--(2.75,.75)--(2,0);
\end{tikzpicture}
\end{center}
Clearly, $X_1$ does not have an element that is comparable with all its elements, however $\ADer(X_1,K)\subseteq \IDer(X_1,K)$, by \cite[Example~3.6]{ER}. On the other hand, $\ADer(X_2,K) \nsubseteq \IDer(X_2,K)$ as seen in \cite[p.257]{SO97} or \cite[Example~3.7]{ER}.

We also recall that $\Mult(FI(X_1,K)) \subseteq \Inn(FI(X_1,K))$, by \cite[Example~7]{BMUL}, but $\Mult(FI(X_2,K)) \nsubseteq \Inn(FI(X_2,K))$ if $\Char K \neq 2$, by \cite[p.1238]{stanley} or \cite[Example~8]{BMUL}. However, $\Mult(FI(X_2,\mathds{Z}_2)) \subseteq \Inn(FI(X_2,\mathds{Z}_2))$, by \cite[Theorem~5]{BMUL}. Therefore, $I(X_2,\mathds{Z}_2)$ has the property that every multiplicative automorphism is inner, but admits an additive derivation which is not inner.

The existence of an element that is comparable with all the elements of $X$ is also sufficient to ensure that every multiplicative automorphism of $FI(X,K)$ is inner, by \cite[Proposition~2.4]{BFS14}.

Let $\eta \in \Aut(FI(X,K))$, $g$ a central unit of $FI(X,K)$ and $D \in \Der(X,K)$. It follows from the proof of \cite[Theorem~4.3]{DW} that
$$\widetilde{\eta} = \begin{bmatrix}
\eta &0\\
0 &\overline{\eta}
\end{bmatrix}\text{,}\quad \widetilde{g} = \begin{bmatrix}
\Id_{FI(X,K)} &0\\
0 & g \cdot
\end{bmatrix} \quad\text{and}\quad
\widetilde{D} = \begin{bmatrix}
\Id_{FI(X,K)} &0\\
D &\Id_{I(X,K)}
\end{bmatrix}$$
are automorphisms of $D(X,K)$, where $g\cdot:I(X,K) \to I(X,K)$ is the multiplication by $g$, and each $\phi \in \Aut(D(X,K))$ has the form
\begin{equation}\label{eqDecDP}
\phi = \widetilde{\eta} \circ \widetilde{g} \circ \widetilde{D}.
\end{equation}
Moreover,  $gD, D_\eta:= \overline{\eta} \circ D \circ \eta^{-1} \in \Der(X,K)$ and $\eta(g)$ is a central unit of $FI(X,K)$ such that
\begin{equation} \label{comutaaut}
\widetilde{\eta} \circ \widetilde{D}  = \widetilde{D_\eta}\circ \widetilde{\eta}\text{,} \quad \widetilde{g}\circ \widetilde{D}  = \widetilde{gD} \circ \widetilde{g} \quad \text{and} \quad \widetilde{\eta}\circ \widetilde{g}  = \widetilde{\eta(g)}\circ  \widetilde{\eta}.
\end{equation}
In the particular case when $D= D_i \in \IDer(X,K)$ and $g = k \delta$ for some $k \in K^*$, we have $(D_i)_{\eta} = D_{\overline{\eta}(i)} \in \IDer(X,K)$ and $\eta(k\delta)= k\delta$, and so
\begin{equation} \label{comutaautparticular}
\widetilde{\eta} \circ \widetilde{D_i}  = \widetilde{D_{\overline{\eta}(i)}}\circ \widetilde{\eta}\quad\text{and}\quad \widetilde{\eta}\circ \widetilde{k\delta}  = \widetilde{k\delta} \circ  \widetilde{\eta}.
\end{equation}

Let $\alpha:X \to X$ be an automorphism and consider the induced automorphism $\widehat{\alpha}:FI(X,K) \to FI(X,K)$. For simplicity, we denote the extension $\overline{\widehat{\alpha}}:I(X,K)\to I(X,K)$ and the automorphism $\widetilde{\widehat{\alpha}}:D(X,K) \to D(X,K)$ by $\overline{\alpha}$ and $\widetilde{\alpha}$, respectively. With this denotations, we have the following.

\begin{proposition} \label{propallcomparableAutDP}
Let $X$ be a connected poset such that $\Mult(FI(X,K)) \subseteq \Inn(FI(X,K))$ and $\Der(X,K) = \IDer(X,K)$ (in particular, when $X$ has an element that is comparable with all its elements). If $\phi \in \Aut(D(X,K))$, then $\phi= \Psi \circ  \widetilde{\alpha} \circ \widetilde{k\delta}$ where $\Psi \in \Inn(D(X,K))$, $\alpha$ is an automorphism of $X$ and $k \in K^{*}$.
\end{proposition}
\begin{proof}
By \eqref{eqDecDP}, $\phi = \widetilde{\eta} \circ \widetilde{g} \circ  \widetilde{D_i}$ for some $i \in I(X,K)$, $\eta \in \Aut(FI(X,K))$ and $g$ a central unit of $FI(X,K)$. Since $X$ is connected, $FI(X,K)$ is a central algebra and, therefore, $g = k\delta$ for some $k \in K^*$. By Theorem~\ref{DecoAutFIP}, $\eta = \Psi_f \circ \widehat{\alpha}$ for some $f \in U(FI(X,K))$, where $\alpha$ is the automorphism of $X$ induced by $\eta$. By Remark~\ref{obsCompoExten}, $\widetilde{\eta} =  \widetilde{\Psi_f \circ \widehat{\alpha}} = \widetilde{\Psi_f} \circ \widetilde{\alpha}$, thus, by \eqref{comutaaut} and \eqref{comutaautparticular}, we have
\begin{align*}
\phi = \widetilde{\eta} \circ \widetilde{k\delta} \circ \widetilde{D_i} = \widetilde{\Psi_f} \circ \widetilde{\alpha} \circ \widetilde{kD_i} \circ \widetilde{k\delta} = \widetilde{\Psi_f} \circ \widetilde{\alpha} \circ \widetilde{D_{ki}} \circ \widetilde{k\delta} =  \widetilde{\Psi_f} \circ \widetilde{D_{\overline{\alpha}(ki)}} \circ  \widetilde{\alpha}\circ \widetilde{k\delta}.
\end{align*}
By \cite[Proposition~7]{DW}, $\Psi:= \widetilde{\Psi_f} \circ \widetilde{D_{\overline{\alpha}(ki)}} \in \Inn(D(X,K))$ and we have $\phi= \Psi \circ  \widetilde{\alpha} \circ \widetilde{k\delta}$.
\end{proof}

\section{Characterization of the anti-automorphisms and involutions on $D(X,K)$}\label{SCharac}

Let $\Phi$ and $\varphi$ be anti-automorphisms (resp.~involutions) of $D(X,K)$ and $FI(X,K)$, respectively. We denote by $\lambda_\Phi$ and $\lambda_\varphi$ the anti-automorphisms (resp.~involutions) on $X$ induced by $\Phi$ and $\varphi$, respectively. We recall that for each $x \in X$, $\Phi \left(\begin{bmatrix}e_x\\0 \end{bmatrix}\right)$ is conjugate to $\begin{bmatrix}e_{\lambda_{\Phi}(x)}\\0 \end{bmatrix}$ and $\varphi(e_x)$ is conjugate to $e_{\lambda_{\varphi}(x)}$, by Theorem~\ref{teoantiisoDP} and \cite[Theorem~3]{BFS12}.

\begin{lemma} \label{lemaphi12}
Let $\Phi = \begin{bmatrix}
\phi_{11} &\phi_{12}\\
\phi_{21} &\phi_{22}
\end{bmatrix}$ be an anti-automorphism of $D(X,K)$. Then $\phi_{11}$ is an anti-automorphism of $FI(X,K)$ such that $\lambda_{\Phi} = \lambda_{\phi_{11}}$, $\phi_{12}=0$ and $\phi_{22} \in \Aut_K(I(X,K))$. Moreover, if $\Phi$ is an involution, then $\phi_{11}$ is an involution on $FI(X,K)$ and $(\phi_{22})^2= \Id_{I(X,K)}$.
\end{lemma}
\begin{proof}
Following the same ideas as in the proof of \cite[Lemma~4.2]{DW}, we can show that $\phi_{12} = 0$. Since $\Phi\left(\begin{bmatrix}\delta\\0 \end{bmatrix}\right) = \begin{bmatrix}\delta\\0 \end{bmatrix}$, we have $\begin{bmatrix}\phi_{11}(\delta)\\ \phi_{21}(\delta) \end{bmatrix} = \begin{bmatrix}\delta\\0 \end{bmatrix}$ and, consequently, $\phi_{11}(\delta)=\delta$. Moreover, for all $f,g \in FI(X,K)$, $\Phi\left(\begin{bmatrix}f\\0 \end{bmatrix}\begin{bmatrix}g\\0 \end{bmatrix}\right) =\Phi\left(\begin{bmatrix}g\\0 \end{bmatrix}\right) \Phi\left(\begin{bmatrix}f\\0 \end{bmatrix}\right)$ implies $\phi_{11}(fg)= \phi_{11}(g)\phi_{11}(f)$. Thus, $\phi_{11}$ is an anti-homomorphism of $FI(X,K)$. From the equalities $\Phi \circ \Phi^{-1}=  \Phi^{-1}\circ\Phi = \Id_{D(X,K)}$  it follows that  $\Phi^{-1}= \begin{bmatrix}
\phi_{11}^{-1} &0\\
\varphi &\phi_{22}^{-1}
\end{bmatrix}$ for some $\varphi \in \Hom_K(FI(X,K),I(X,K))$. Therefore, $\phi_{11}$ is an anti-automorphism of $FI(X,K)$ and $\phi_{22} \in \Aut_K(I(X,K))$. Moreover, if $\Phi \left(\begin{bmatrix}e_x\\0 \end{bmatrix}\right) =\begin{bmatrix}g_x\\i_x \end{bmatrix}\begin{bmatrix}e_{\lambda_{\Phi}(x)}\\0 \end{bmatrix}\begin{bmatrix}g_x\\i_x \end{bmatrix}^{-1}$, then $\phi_{11}(e_x) = g_x e_{\lambda_{\Phi}(x)} g_x^{-1}$ and, therefore, $\lambda_{\phi_{11}}= \lambda_{\Phi}$.

If $\Phi$ is an involution on $D(X,K)$, then $\Phi^{-1}=\Phi$. Thus $\phi_{11}$ is an involution of $FI(X,K)$ and $(\phi_{22})^2 = \Id_{I(X,K)}$.
\end{proof}

Let $\rho$ be an anti-automorphism of $FI(X,K)$ and $\overline{\rho}: I(X,K) \to I(X,K)$ as in Proposition~\ref{propesten}. It is easy to check that $\widetilde{\rho}=  \begin{bmatrix} \rho &0\\ 0 &\overline{\rho} \end{bmatrix}$ is an anti-automorphism of $D(X,K)$ such that ($\widetilde{\rho})^{-1}= \widetilde{\rho^{-1}}$. In particular, if $\rho$ is an involution, so is $\widetilde{\rho}$.

\begin{theorem} \label{teoDecinvDP}
Let $\Phi = \begin{bmatrix} \rho &0\\ \phi_{21} &\phi_{22} \end{bmatrix}$ be an anti-automorphism (involution) of $D(X,K)$. Then $\Phi = \widetilde{\rho}\circ \widetilde{g}\circ \widetilde{D}$ where $\rho$ is an anti-automorphism (involution) of $FI(X,K)$ such that $\lambda_{\Phi} = \lambda_{\rho}=\lambda_{\widetilde{\rho}}$, $g$ is a central unit of $FI(X,K)$ and $D \in \Der(X,K)$.
\end{theorem}
\begin{proof}
By Lemma~\ref{lemaphi12}, $\rho$ is an anti-automorphism (involution) of $FI(X,K)$ such that $\lambda_{\Phi} = \lambda_{\rho}=\lambda_{\widetilde{\rho}}$.   Consider the anti-automorphism (involution) $\widetilde{{\rho}^{-1}}$ of $D(X,K)$. Then $\widetilde{\rho^{-1}}\circ \Phi \in \Aut(D(X,K))$. By  \eqref{eqDecDP}, there are $\eta \in \Aut(FI(X,K))$, $D \in \Der(X,K)$  and $g$ a central unit of $FI(X,K)$ such that $\widetilde{\rho^{-1}}\circ \Phi = \widetilde{\eta}\circ \widetilde{g}\circ \widetilde{D}$. Since
$$\widetilde{\rho^{-1}}\circ \Phi  = \begin{bmatrix}
\Id_{FI(X,K)} &0\\
(\overline{\rho})^{-1}\circ\phi_{21} &(\overline{\rho})^{-1}\circ\phi_{22}
\end{bmatrix}
\text{ and }\widetilde{\eta}\circ \widetilde{g}\circ \widetilde{D} =   \begin{bmatrix}
\eta &0\\
\overline{\eta} \circ g\cdot\circ D &\overline{\eta} \circ g\cdot
\end{bmatrix},$$
we have $\eta = \Id_{FI(X,K)}$ and, by Lemma~\ref{extId}, $\overline{\eta}= \Id_{I(X,K)}$. Thus, $\widetilde{\eta}= \Id_{D(X,K)}$ and, therefore, $\Phi = \widetilde{\rho}\circ \widetilde{g}\circ \widetilde{D}$.
\end{proof}

\begin{remark}
Let $\rho$ be an anti-automorphism of $FI(X,K)$, $g \in FI(X,K)$  a central unit and $D \in \Der(X,K)$. As for automorphisms, we denote $D_\rho = \overline{\rho}\circ D \circ \rho^{-1}$. It is easy to show that $D_\rho \in \Der(X,K)$ and
\begin{equation} \label{comutaantiaut}
\widetilde{\rho} \circ \widetilde{D}  = \widetilde{D_\rho}\circ \widetilde{\rho}\quad\text{and}\quad \widetilde{\rho}\circ \widetilde{g}  = \widetilde{\rho(g)}\circ  \widetilde{\rho}.
\end{equation}
In the particular case when $D=D_i \in \IDer(X,K)$ and $g= k\delta$ for some $k \in K^*$, we have
\begin{equation} \label{comutaantiautparticular}
\widetilde{\rho} \circ \widetilde{D_i}  = \widetilde{D_{-\overline{\rho}(i)}} \circ \widetilde{\rho} \quad \text{and} \quad \widetilde{\rho}\circ \widetilde{k\delta}  = \widetilde{k\delta} \circ  \widetilde{\rho}.
\end{equation}
\end{remark}

\begin{corollary} \label{coroallcomparableAntiauttDP}
Let $X$ be a connected poset such that $\Der(X,K)=\IDer(X,K)$ and $\Mult(FI(X,K)) \subseteq \Inn(FI(X,K))$ (in particular, when $X$ has an element that is comparable with all its elements). If $\Phi$ is an anti-automorphism (involution) of $D(X,K)$, then $\Phi= \Psi \circ  \widetilde{\rho_\lambda} \circ \widetilde{k\delta}$ where $\Psi \in \Inn(D(X,K))$, $\lambda$ is the anti-automorphism (involution) of $X$ induced by $\Phi$ and $k \in K^{*}$. Moreover, when $\Phi$ is an involution, $k=\pm 1$.
\end{corollary}
\begin{proof}
The proof that $\Phi= \Psi \circ \widetilde{\rho_\lambda} \circ \widetilde{k\delta}$ where $\Psi\in \Inn(D(X,K))$, $\lambda$ is the anti-automorphism (involution) of $X$ induced by $\Phi$ and $k \in K^{*}$ is analogous to the proof of Proposition~\ref{propallcomparableAutDP}, just replacing \eqref{eqDecDP} by Theorem~\ref{teoDecinvDP} and the equalities \eqref{comutaaut} and \eqref{comutaautparticular} by \eqref{comutaantiaut} and \eqref{comutaantiautparticular}, respectively.

Suppose $\Phi$ is an involution. We have
\begin{align*}
\Phi \left( \begin{bmatrix}0\\\delta \end{bmatrix} \right)
&= (\Psi \circ \widetilde{\rho_\lambda} \circ \widetilde{k \delta}) \left( \begin{bmatrix}0\\ \delta \end{bmatrix} \right)
= ( \Psi \circ \widetilde{\rho_\lambda}) \left( \begin{bmatrix}0\\ k\delta \end{bmatrix} \right)
\\&=   \Psi  \left( \begin{bmatrix}0\\ \overline{\rho_\lambda}(k\delta)\end{bmatrix} \right)
=  \Psi  \left( \begin{bmatrix}0\\ k\delta\end{bmatrix} \right)= \begin{bmatrix}0\\ k\delta \end{bmatrix},
\end{align*}
since $\begin{bmatrix}0\\ k\delta \end{bmatrix} \in Z(D(X,K))$. Thus, $$\begin{bmatrix} 0 \\ \delta \end{bmatrix} = \Phi^{2}\left( \begin{bmatrix}0\\\delta \end{bmatrix} \right) = \Phi \left( k\begin{bmatrix}0\\ \delta \end{bmatrix} \right) = k \begin{bmatrix}0\\ k\delta \end{bmatrix} = \begin{bmatrix}0\\ k^2\delta \end{bmatrix}$$
and, therefore, $k=\pm 1$.
\end{proof}

\section{Classification of involutions on $D(X,K)$}\label{Sclassifi}

Let $\phi$ and $\varphi$ be involutions on a $K$-algebra $A$. We recall that $\phi$ and $\varphi$ are \emph{equivalent} if there is $\Phi \in \Aut(A)$ such that $\Phi\circ \phi = \varphi \circ \Phi$.

In this section we give the classification of involutions on $D(X,K)$ in the case when $X$ is a connected poset such that $\Mult(FI(X,K)) \subseteq \Inn(FI(X,K))$ and $\Der(X,K) = \IDer(X,K)$, and $K$ is a field of characteristic different from $2$.

We start by looking for necessary conditions for two involutions on $D(X,K)$ to be equivalent (via inner automorphisms).

\begin{remark} \label{obsInnPsiepsitiu}
Let $\theta = \begin{bmatrix} f \\ j\end{bmatrix} \in U(D(X,K))$. It is easy to prove (see proof of \cite[Proposition~7]{DW}) that $\Psi_\theta = \widetilde{\Psi_f} \circ \widetilde{D_{-f^{-1}j}}$.  In particular, we have $$\widetilde{\Psi_f} = \Psi_{{\scriptsize\begin{bmatrix}f\\0 \end{bmatrix}}}= \Psi_{{\scriptsize\begin{bmatrix}f\\f \end{bmatrix}}} = \Psi_{{\scriptsize\begin{bmatrix}f\\-f \end{bmatrix}}} = \Psi_{{\scriptsize \begin{bmatrix}-f\\f \end{bmatrix}}}.$$
\end{remark}

\begin{proposition} \label{propeqivalenciaAutandIAut}
Let $\Phi_1 =\begin{bmatrix} \rho_1 &0\\ \phi_{21} &\phi_{22} \end{bmatrix}$, $\Phi_2 =\begin{bmatrix}
\rho_2 &0\\ \varphi_{21} &\varphi_{22} \end{bmatrix}$ be involutions on $D(X,K)$. If there is $\Psi \in \Aut(D(X,K))$ $  (\Inn(D(X,K)))$ such that $\Psi\circ \Phi_1 = \Phi_2 \circ \Psi$, then there is $\eta \in \Aut(FI(X,K))$ $(\Inn(FI(X,K)))$ such that $\eta\circ\rho_1=\rho_2 \circ\eta$.
\end{proposition}
\begin{proof}
Suppose there is $\Psi \in \Aut(D(X,K))$  such that $\Psi \circ \Phi_1= \Phi_2 \circ \Psi$. By \cite[Lemma~4.2]{DW} and the paragraph that follows it, $\Psi = \begin{bmatrix}
\eta &0\\
\psi_{21} &\psi_{22}
\end{bmatrix}$, where $\eta  \in \Aut(FI(X,K))$. Then
\begin{equation}\label{eqpequivalencia}
\begin{bmatrix}
\eta\circ\rho_1 &0\\
\psi_{21}\circ \rho_1 + \psi_{22}\circ \phi_{21}&\psi_{22}\circ\phi_{22}
\end{bmatrix} = \begin{bmatrix}
\rho_2 \circ \eta &0\\
\varphi_{21}\circ \eta + \varphi_{22}\circ \psi_{21}&\varphi_{22}\circ\psi_{22}
\end{bmatrix}
\end{equation}
which implies $\eta\circ\rho_1=\rho_2 \circ\eta$.

If $\Psi =\Psi_{{\scriptsize\begin{bmatrix}f\\j \end{bmatrix}}} \in \Inn(D(X,K))$, then $\Psi = \widetilde{\Psi_f} \circ \widetilde{D_{-f^{-1}j}} = \begin{bmatrix}
\Psi_f &0\\
\overline{\Psi_f} \circ D_{-f^{-1}j}  &\overline{\Psi_f}
\end{bmatrix}$, by Remark~\ref{obsInnPsiepsitiu}. Therefore, $\eta = \Psi_f \in \Inn(FI(X,K))$.
\end{proof}

\begin{corollary}  \label{coroinduzesinal}
Let $\Phi_1$ and $\Phi_2$ be involutions on $D(X,K)$ such that $\Phi_i =\Psi_{\theta_i} \circ \widetilde{\rho_{\lambda_i}} \circ \widetilde{k_i\delta}$ where $\theta_i \in U(D(X,K))$, $\lambda_i$ is an involution on $X$, and $k_i \in \{-1,1\}$ for $i=1,2$. If there is $\Psi \in \Inn(D(X,K))$ such that $\Psi \circ \Phi_1 =\Phi_2 \circ \Psi$, then $\lambda_1 = \lambda_2$ $(=\lambda_{\Phi_1}= \lambda_{\Phi_2})$ and $k_1=k_2$.
\end{corollary}
\begin{proof}
Firstly, for each $x \in X$ we have
$$\Phi_i \left(\begin{bmatrix}e_x\\ 0 \end{bmatrix}\right) = (\Psi_{\theta_i} \circ \widetilde{\rho_{\lambda_i}}) \left(\begin{bmatrix}e_x\\ 0 \end{bmatrix}\right)=\Psi_{\theta_i} \left(\begin{bmatrix} \rho_{\lambda_i}(e_x)\\ 0 \end{bmatrix}\right)= \theta_i\begin{bmatrix} e_{\lambda_i(x)}\\ 0 \end{bmatrix}\theta_i^{-1},$$
thus $\lambda_{\Phi_i}=\lambda_i$, $i=1,2$.

Suppose there is $\Psi \in \Inn(D(X,K))$ such that $\Psi \circ \Phi_1= \Phi_2 \circ \Psi$. By Proposition~\ref{propeqivalenciaAutandIAut} there is $\eta \in \Inn(FI(X,K))$ such that $\eta \circ \rho_1 = \rho_2 \circ \eta$, where $\Phi_1 =\begin{bmatrix} \rho_1 &0\\ \phi_{21} &\phi_{22} \end{bmatrix}$ and $\Phi_2 =\begin{bmatrix} \rho_2 &0\\ \varphi_{21} &\varphi_{22} \end{bmatrix}$. Thus, by Proposition~\ref{propAutInner} (iv), there is $u \in U(FI(X,K))$ such that $\rho_1 = \Psi_u \circ \rho_2$. Now, by \cite[Theorem~3]{BFS12}, $\rho_i(e_x)$ is conjugate to $e_{\lambda_{\rho_i}(x)}$ for each $x \in X$, $i=1,2$. Let $g_x \in U(FI(X,K))$ such that $\rho_2(e_x)= g_xe_{\lambda_{\rho_2}(x)} g_x^{-1}$. Then
$$\rho_1(e_x) = (\Psi_u \circ \rho_2)(e_x) = ug_xe_{\lambda_{\rho_2}(x)} g_x^{-1}u^{-1}$$
which implies $\lambda_{\rho_1}(x) = \lambda_{\rho_2}(x)$, for each $x \in X$. Therefore, $\lambda_{\Phi_1}=\lambda_{\rho_1}=\lambda_{\rho_2}=\lambda_{\Phi_2}$, by Theorem~\ref{teoDecinvDP}.

To show that $k_1=k_2$, observe that $\widetilde{k_i\delta}= (\widetilde{k_i\delta})^{-1}$, since $k_i^2=1$, $i=1,2$. By Proposition~\ref{propAutInner} (iv), there is $\Psi_1 \in \Inn(D(X,K))$ such that $\Phi_2 = \Psi_1\circ \Phi_1$. Thus, by (\ref{comutaantiautparticular}),
\begin{align*}
\Psi_{\theta_2} \circ \widetilde{\rho_{\lambda_1}} \circ \widetilde{k_2\delta} = \Psi_1\circ \Psi_{\theta_1} \circ \widetilde{\rho_{\lambda_1}} \circ \widetilde{k_1\delta} &\Rightarrow  \Psi_{\theta_2}  \circ \widetilde{k_2\delta} \circ \widetilde{\rho_{\lambda_1}}= \Psi_1\circ \Psi_{\theta_1}  \circ \widetilde{k_1\delta}\circ \widetilde{\rho_{\lambda_1}} \\&\Rightarrow \Psi_{\theta_2}^{-1}  \circ \Psi_1 \circ \Psi_{\theta_1}  =  \widetilde{k_2\delta} \circ \widetilde{k_1\delta} = \widetilde{k_2k_1\delta}.
\end{align*}
By \cite[Theorem~4.3 and Proposition~7]{DW}, $\widetilde{k_2k_1\delta} = \Id_{D(X,K)}$ and, therefore, $k_1 =k_2$.
\end{proof}

\begin{corollary} \label{coroSinal}
Let $\Phi$ be an involution on $D(X,K)$ such that $\Phi=  \Psi_{\theta_1} \circ \widetilde{\rho_{\lambda_1}} \circ \widetilde{k_1\delta} =  \Psi_{\theta_2} \circ \widetilde{\rho_{\lambda_2}} \circ \widetilde{k_2\delta}$ where $\theta_i \in U(D(X,K))$, $\lambda_i$ is an involution on $X$ and $k_i \in \{-1,1\}$, $i=1,2$. Then $\lambda_1 = \lambda_2 = \lambda_{\Phi}$, $k_1 = k_2$ and $\theta_1 =  c\theta_2$ for some $c \in Z(D(X,K))$.
\end{corollary}
\begin{proof}
It follows directly from Corollary~\ref{coroinduzesinal} and Proposition~\ref{propAutInner} (i).
\end{proof}

\begin{definition}
Let $\Phi$ be an involution on $D(X,K)$ such that $\Phi= \Psi \circ \widetilde{\rho_\lambda}\circ \widetilde{k\delta}$, where $\Psi \in \Inn(D(X,K))$, $\lambda$ is the involution on $X$ induced by $\Phi$ and $k \in \{ -1, 1\}$. We call $k$  the \emph{sign} of $\Phi$ and we denote it by $s(\Phi)$.
\end{definition}

Let $\rho$ be an anti-automorphism of $FI(X,K)$. Then $\widetilde{\rho} \circ \widetilde{-\delta}= \begin{bmatrix} \rho &0\\ 0 &-\overline{\rho}
\end{bmatrix}$ is an anti-automorphism of $D(X,K)$, since $\widetilde{-\delta} \in \Aut(D(X,K))$. Moreover, if $\rho$ is an involution, so is $\widetilde{\rho} \circ \widetilde{-\delta}$, by Corollary~\ref{coroexteninvo}.

\begin{proposition} \label{propequivalenciainner}
Let $X$ be a poset and let $\rho_1$, $\rho_2$ be involutions on $FI(X,K)$. The following statements are equivalent:
\begin{enumerate}
\item[(i)] $\rho_1$ and $\rho_2$ are equivalent (via inner automorphism);
\item[(ii)] $\widetilde{\rho_1}$ and $\widetilde{\rho_2}$ are equivalent (via inner automorphism);
\item[(iii)] $\widetilde{\rho_1} \circ \widetilde{-\delta}$ and $\widetilde{\rho_2} \circ \widetilde{-\delta}$ are equivalent (via inner automorphism).
\end{enumerate}
\end{proposition}
\begin{proof}
Since $\widetilde{\rho_i} = \begin{bmatrix}
 \rho_i &0\\
0  &\overline{\rho_i}
\end{bmatrix}$ and $\widetilde{\rho_i} \circ \widetilde{-\delta} = \begin{bmatrix}
 \rho_i &0\\
0  &-\overline{\rho_i}
\end{bmatrix}$, $i=1,2$, we have (ii)~$\Rightarrow$~(i) and (iii)~$\Rightarrow$~(i), by Proposition~\ref{propeqivalenciaAutandIAut}.

Now, suppose there is $\eta \in \Aut(FI(X,K))$ such that $\eta \circ \rho_1 = \rho_2 \circ \eta$. By Remark~\ref{obsCompoExten} and \eqref{comutaantiautparticular},
\begin{align*}
\widetilde{\eta} \circ (\widetilde{\rho_1} \circ \widetilde{k\delta}) &= \widetilde{\eta \circ \rho_1} \circ \widetilde{k\delta} = \widetilde{k\delta}\circ\widetilde{\eta \circ \rho_1} = \widetilde{k\delta}\circ\widetilde{\rho_2 \circ \eta} \\&= \widetilde{k\delta}\circ\widetilde{\rho_2} \circ \widetilde{\eta} = (\widetilde{\rho_2} \circ \widetilde{k\delta}) \circ \widetilde{\eta},
\end{align*}
for any $k \in K^*$. Moreover, if $\eta \in \Inn(FI(X,K))$, then $\widetilde{\eta}\in \Inn(D(X,K))$, by Remark~\ref{obsInnPsiepsitiu}. Therefore, (i) $\Rightarrow$ (ii) and (iii).
\end{proof}

\begin{proposition} \label{propsinalcentro}
Let $X$ be a connected poset such that $\Mult(FI(X,K)) \subseteq \Inn(FI(X,K))$ and $\Der(X,K) = \IDer(X,K)$. If $\Phi$ is an involution on $D(X,K)$, then $\Phi (c_{k_1,k_2}) =  c_{k_1,s(\Phi)k_2}$ for all $k_1, k_2 \in K$.
\end{proposition}
\begin{proof}
By Corollaries~\ref{coroallcomparableAntiauttDP} and \ref{coroSinal}, $\Phi = \Psi \circ \widetilde{\rho_{\lambda_{\Phi}}} \circ \widetilde{s(\Phi)\delta}$, for some $\Psi \in \Inn(D(X,K))$. Therefore, $$\Phi (c_{k_1,k_2}) = \Psi( \widetilde{\rho_{\lambda_{\Phi}}}(  c_{k_1,s(\phi)k_2}))= \Psi(  c_{k_1,s(\phi)k_2}) =  c_{k_1,s(\phi)k_2}.$$
\end{proof}

\begin{proposition} \label{propmaisoummenosTheta}
Let $X$ be a connected poset such that $\Mult(FI(X,K)) \subseteq \Inn(FI(X,K))$ and $\Der(X,K) = \IDer(X,K)$. Let $\Phi_1$,$\Phi_2$ be involutions on $D(X,K)$ such that $\Phi_2 = \Psi_{\theta}\circ \Phi_1$ for some $\theta \in U(D(X,K))$. Then $\Phi_1(\theta) = c_{k_0,k_1}\theta$ for some $k_0 \in \{-1,1\}$ and $k_1 \in K$. Moreover, when $\Char K \neq 2$,
\begin{enumerate}
\item[(i)] if $s(\Phi_1)=1$, then $k_1=0$ and $\Phi_1(\theta) = k_0 \theta$;
\item[(ii)] if $s(\Phi_1)= -1$, then there is $\gamma \in U(D(X,K))$ such that $\Phi_2 = \Psi_\gamma \circ \Phi_1$ and $\Phi_1(\gamma) = k_0\gamma$.
\end{enumerate}
\end{proposition}
\begin{proof}
By Proposition~\ref{propAutInner} (ii), $\Psi_{\theta}= \Psi_{\Phi_1(\theta)}$ and, consequently, $\Phi_1(\theta) = c_{k_0,k_1}\theta$ for some $k_0 \in K^*$ and $k_1 \in K$, by Proposition~\ref{propAutInner} (i) and Proposition~\ref{propCenDXKconnected}. Thus, by Proposition~\ref{propsinalcentro},
\begin{align*}
\theta &= \Phi_1^2(\theta) = \Phi_1 \left( c_{k_0,k_1}\theta \right) = \Phi_1 \left( c_{k_0,k_1}\right) \Phi_1(\theta) = c_{k_0,s(\Phi_1)k_1} c_{k_0,k_1}\theta \\&= \begin{bmatrix}k_0\delta\\  s(\Phi_1) k_1\delta\end{bmatrix} \begin{bmatrix}k_0\delta\\ k_1\delta\end{bmatrix}\theta= \begin{bmatrix}(k_0)^2\delta\\ k_0 k_1 \delta + s(\Phi_1) k_0 k_1\delta\end{bmatrix} \theta.
\end{align*}
Therefore,
\begin{equation} \label{eqPropmaisoumenostheta}
\begin{bmatrix}(k_0)^2\delta\\ k_0 k_1 \delta + s(\Phi_1) k_0 k_1\delta\end{bmatrix} = \begin{bmatrix}\delta \\ 0\end{bmatrix}
\end{equation} and so $k_0  \in\{-1,1\}$. Now, suppose $\Char K\neq 2$.

(i) If $s(\Phi_1) = 1$, then $ 2 k_0 k_1\delta = 0$, by \eqref{eqPropmaisoumenostheta}. Thus, $k_1=0$ and so $\Phi_1(\theta) = k_0 \theta$.

(ii) Suppose  $s(\Phi_1)=-1$ and consider the element $\gamma = c_{k_0,k_1/2}\theta \in U(D(X,K))$. By Proposition~\ref{propAutInner} (i), $\Psi_\gamma = \Psi_\theta$ since $c_{k_0,k_1/2} \in Z(D(X,K))$, thus $\Phi_2 = \Psi_\theta \circ \Phi_1 = \Psi_\gamma \circ \Phi_1$. Moreover, by Proposition~\ref{propsinalcentro},
\begin{align*}
\Phi_1(\gamma) &= \Phi_1(c_{k_0,k_1/2}) \Phi_1(\theta) = c_{k_0,s(\phi_1)k_1/2} c_{k_0,k_1}\theta = c_{k_0,-k_1/2} c_{k_0,k_1}\theta
\\&=  \begin{bmatrix}k_0\delta \\ (-k_1/2)\delta\end{bmatrix} \begin{bmatrix}k_0\delta \\ k_1\delta\end{bmatrix} \theta =
\begin{bmatrix}k_0^2\delta \\ (k_0k_1-k_0k_1/2)\delta\end{bmatrix}\theta =\begin{bmatrix}k_0^2\delta \\ (k_0k_1/2)\delta\end{bmatrix}\theta
\\&=k_0 \begin{bmatrix}k_0\delta \\ (k_1/2)\delta\end{bmatrix}\theta
= k_0 c_{k_0,k_1/2}\theta  = k_0\gamma.
\end{align*}
\end{proof}

\subsection{Classification via inner automorphisms}\label{SclassiInner}

We first present the classification of involutions on $D(X,K)$ via inner automorphisms. We then use this classification to obtain the general classification of involutions.

From now on, $K$ is a field of characteristic different from $2$ and $X$ is a connected poset such that $\Mult(FI(X,K)) \subseteq \Inn(FI(X,K))$ and $\Der(X,K) = \IDer(X,K)$. In this case, each involution on $D(X,K)$ has the form $\Phi= \Psi \circ \widetilde{\rho_{\lambda_{\Phi}}} \circ \widetilde{s(\Phi)\delta}$ for some $\Psi \in \Inn(D(X,K))$, by Corollaries~\ref{coroallcomparableAntiauttDP} and \ref{coroSinal}.

By Corollaries~\ref{coroallcomparableAntiauttDP} and \ref{coroinduzesinal}, two involutions on $D(X,K)$ that induce different involutions on $X$ are not equivalent via inner automorphims. For that, we fix an involution $\lambda$ on $X$.

By \cite[Theorem~4.7]{BFS14}, there is a triple of disjoint subsets $(X_1, X_2, X_3)$ of $X$ with $X = X_1 \cup X_2 \cup X_3$ satisfying:

(i) $X_3=\{ x \in X : \lambda(x) = x\}$;

(ii) if $x \in X_1$ ($X_2$), then $\lambda(x) \in X_2$ ($X_1$);

(iii) if $x \in X_1$ ($X_2$) and $y \leq x$ ($x \leq y$), then $y \in X_1$ ($X_2$).

Let $\epsilon: X_3 \to K^*$ be a map and consider $u_\epsilon \in U(FI(X,K))$ defined by
$$
u_\epsilon(x,y) = \begin{cases}
0 &\text{if } x \neq y\\
1 &\text{if } x= y \in X_1\cup X_2\\
\epsilon(x) &\text{if } x= y \in X_3
\end{cases}.
$$
As seen in \cite[p.1096]{BFS14}, $\rho_\lambda(u_\epsilon) = u_\epsilon$ and $\rho_\epsilon = \Psi_{u_\epsilon} \circ \rho_\lambda$ is an involution on $FI(X,K)$ such that $\rho_\epsilon = \rho_\lambda$ when $X_3 = \emptyset$. Thus, $\rho_\epsilon(u_\epsilon)=u_\epsilon$ and so
\begin{equation} \label{eqrhoepsthetaeps1}
    \widetilde{\rho_\epsilon} (\theta_\epsilon)= \theta_\epsilon, \text{ where } \theta_\epsilon = \begin{bmatrix}u_\epsilon\\0 \end{bmatrix}.
\end{equation}
Moreover, by Remarks~\ref{obsCompoExten} and \ref{obsInnPsiepsitiu},
\begin{equation} \label{eqrhoepsthetaeps2}
    \widetilde{\rho_\epsilon}  = \Psi_{\theta_\epsilon} \circ \widetilde{\rho_\lambda}.
\end{equation}

\begin{lemma} \label{lemaiffsimetricoo}
Let $\epsilon: X_3 \to K^*$ be a map, $\theta = \begin{bmatrix}f\\i \end{bmatrix} \in U(D(X,K))$ and $k \in \{-1,1\}$. Suppose $\theta$ is $(\widetilde{\rho_{\epsilon}}\circ \widetilde{k\delta})$-symmetric. Then $\theta = \gamma (\widetilde{\rho_{\epsilon}}\circ \widetilde{k\delta})(\gamma)$ for some $\gamma \in U(D(X,K))$ if and only if $f(x,x) \in (K^*)^{2}$, for all $x \in X_3$.
\end{lemma}
\begin{proof}
Suppose $f(x,x) \in (K^*)^2$ for all $x \in X_3$. Since $\theta$ is $(\widetilde{\rho_{\epsilon}}\circ \widetilde{k\delta})$-symmetric,
\begin{equation} \label{eqsimetrico}
\rho_{\epsilon}(f) =f \text{ and } k\overline{\rho_{\epsilon}}(i) =i.
 \end{equation}
For $x \leq y$ in $X$ there are six possibilities:
\begin{enumerate}[(1)]
\item Both are in $X_1$;
\item Both are in $X_2$;
\item $x \in X_1$ and $y \in X_2$;
\item $x \in X_1$  and  $y \in X_3$;
\item $x \in X_3$ and $y \in X_2$;
\item $x = y \in X_3$.
\end{enumerate}
So we define $v, j \in I(X,K)$ as follows:
$$v(x,y)=\begin{cases}
\delta_{xy} &\text{if } x , y \in X_1\\
f(x,y) &\text{if } x , y \in X_2\\
f(x,y)/2 &\text{if } x \in X_1 \text{ and } y \in X_2\\
0 &\text{if } x \in X_1 \text{ and } y \in X_3\\
f(x,y) &\text{if } x \in X_3 \text{ and } y \in X_2\\
\sqrt{f(x,x)} &\text{if } x = y \in X_3
\end{cases},$$
$$j(x,y)= \begin{cases}
0 &\text{if } x , y \in X_1\\
i(x,y) &\text{if } x , y \in X_2\\
i(x,y)/2 &\text{if } x \in X_1 \text{ and } y \in X_2\\
0 &\text{if } x \in X_1 \text{ and } y \in X_3\\
i(x,y) &\text{if } x \in X_3 \text{ and } y \in X_2\\
\dfrac{i(x,x)}{2\sqrt{f(x,x)} }&\text{if } x = y \in X_3
\end{cases}.$$
Clearly $v \in U(FI(X,K))$, and it follows from the proof of \cite[Lemma~5.3]{BFS14} that $f = v \rho_\epsilon(v)$. Consider $\gamma = \begin{bmatrix}v\\j \end{bmatrix} \in U(D(X,K))$.  We will prove that $\theta = \gamma (\widetilde{\rho_{\epsilon}}\circ \widetilde{k\delta})(\gamma)$. For that, it remains to show that
\begin{equation}\label{eqigualdades}
 i = kv\overline{\rho_{\epsilon}}(j) + j\rho_{\epsilon}(v).
\end{equation}

Let $x \leq  y \in X$. We shall verify that $i(x,y)= [kv\overline{\rho_{\epsilon}}(j) + j\rho_{\epsilon}(v)](x,y)$ for each of the six possibilities above. We have
\begin{equation} \label{eqsomatorio}
    [kv\overline{\rho_{\epsilon}}(j) + j\rho_{\epsilon}(v)](x,y) = k\sum\limits_{x \leq t \leq y}v(x,t)\overline{\rho_{\epsilon}}(j)(t,y) +  \sum\limits_{x \leq t \leq y} j(x,t)\rho_{\epsilon}(v)(t,y)
\end{equation}
and, for any $l \in I(X,K)$,
\begin{equation} \label{eqrhoepsappinL}
\overline{\rho_{\epsilon}}(l)(x,y) = u_\epsilon(x,x)l(\lambda(y),\lambda(x))u_{\epsilon}^{-1}(y,y).
\end{equation}

(1. $x,y \in X_1$) In this case,  $\lambda(x), \lambda(y) \in X_2$ and if $t \in X$ is such that $x \leq t \leq y$, then $t \in X_1$. Thus, $v(x,t) = 0$ for all $x<t\leq y$ and $j(x,t) = 0$ for all $x\leq t\leq y$. Then by \eqref{eqsimetrico}, \eqref{eqsomatorio} and \eqref{eqrhoepsappinL}, we have
\begin{align*}
[kv\overline{\rho_{\epsilon}}(j) + j\rho_{\epsilon}(v)](x,y) &= kv(x,x)\overline{\rho_{\epsilon}}(j)(x,y)  = k u_\epsilon(x,x)j(\lambda(y), \lambda(x)) u_{\epsilon}^{-1}(y,y) \\& = ku_\epsilon(x,x)i(\lambda(y), \lambda(x)) u_{\epsilon}^{-1}(y,y) = k\overline{\rho_{\epsilon}}(i)(x,y) = i(x,y).
\end{align*}

(2. $x,y \in X_2$) Now, we have $\lambda(x), \lambda(y) \in X_1$ and if  $t \in X$ is such that $x\leq t \leq y$, then $t \in X_2$ and $\lambda(t) \in X_1$. Thus, if $x\leq t < y$, then $v(\lambda(y), \lambda(t)) = 0$ and, consequently, $\rho_{\epsilon}(v)(t,y) = 0$, by \eqref{eqrhoepsappinL}. Moreover, if $x\leq t \leq y$, then $j(\lambda(y), \lambda(t)) = 0$ and, by \eqref{eqrhoepsappinL}, $\overline{\rho_{\epsilon}}(j)(t,y) = 0$. It follows from \eqref{eqsomatorio} and \eqref{eqrhoepsappinL} that
\begin{align*}
[kv\overline{\rho_{\epsilon}}(j) + j\rho_{\epsilon}(v)](x,y) &= j(x,y)\rho_{\epsilon}(v)(y,y)  = i(x,y) v(\lambda(y), \lambda(y))  = i(x,y).
\end{align*}

(3. $x\in X_1$, $y\in X_2$) In this case, $j(x,x)=0$ and $\overline{\rho_{\epsilon}}(j)(y,y)= j(\lambda(y),\lambda(y))=0$, because $\lambda(y) \in X_1$. If $x<t <y$, then $t \in X_1 \cup X_2 \cup X_3$.  If $t \in X_2$ then $\lambda(t) \in X_1$. Therefore, $\overline{\rho_{\epsilon}}(j)(t,y)= j(\lambda(y),\lambda(t))=0$ and $\rho_{\epsilon}(v)(t,y) = v(\lambda(y),\lambda(t))=0$. On the other hand, if $t \in X_1 \cup X_3$ then $v(x,t) = 0$ and $j(x,t)=0$. Thus, by \eqref{eqsimetrico}, \eqref{eqsomatorio} and \eqref{eqrhoepsappinL},
 \begin{align*}
 [kv\overline{\rho_{\epsilon}}(j) + j\rho_{\epsilon}(v)](x,y)
 &= kv(x,x)\overline{\rho_{\epsilon}}(j)(x,y) + j(x,y)\rho_{\epsilon}(v)(y,y)
 \\&= k \overline{\rho_{\epsilon}}(j)(x,y) + j(x,y)v(\lambda(y),\lambda(y))
 \\& = kj(\lambda(y),\lambda(x)) +j(x,y)= \dfrac{ ki(\lambda(y),\lambda(x))}{2}+ \dfrac{i(x,y)}{2}
 \\&= \dfrac{1}{2}[  ku_\epsilon(x,x)i(\lambda(y),\lambda(x))u_{\epsilon}^{-1}(y,y) + i(x,y)] \\&= \dfrac{1}{2}[ k\overline{\rho_{\epsilon}}(i)(x,y)   + i(x,y)] = i(x,y).
\end{align*}

(4. $x\in X_1$, $y\in X_3$) If $x\leq t \leq y$, then $t \in X_1 \cup X_3$. By definition,  $v(x,t) =0$ for $x < t \leq y$ and $j(x,t) = 0$ for $x\leq t \leq y$. By \eqref{eqsimetrico}, \eqref{eqsomatorio} and \eqref{eqrhoepsappinL},
 \begin{align*}
 [kv\overline{\rho_{\epsilon}}(j) + j\rho_{\epsilon}(v)](x,y) &= kv(x,x)\overline{\rho_{\epsilon}}(j)(x,y)=  kj(\lambda(y),\lambda(x))[\epsilon(y)]^{-1}\\&= ki(\lambda(y),\lambda(x))[\epsilon(y)]^{-1}= ku_\epsilon(x,x)i(\lambda(y),\lambda(x))u_{\epsilon}^{-1}(y,y) \\&= k\overline{\rho_{\epsilon}}(i)(x,y) = i(x,y).
 \end{align*}

(5. $x\in X_3$, $y\in X_2$) If $x\leq t \leq y$, then $t \in X_2 \cup X_3$. So $\lambda(t) \in  X_1 \cup X_3$  and $\lambda(y) \in X_1$, therefore $\overline{\rho_{\epsilon}}(j)(t,y) = 0$, by \eqref{eqrhoepsappinL}. Moreover, $\rho_{\epsilon}(v)(t,y) = 0$ for $x \leq t < y$. Thus,
 \begin{align*}
 [kv\overline{\rho_{\epsilon}}(j) + j\rho_{\epsilon}(v)](x,y) &= j(x,y)\rho_{\epsilon}(v)(y,y)=  j(x,y) v(\lambda(y),\lambda(y))= i(x,y).
 \end{align*}

(6. $x=y \in X_3$) If $k=-1$, then $i(x,x) = 0$, by \eqref{eqsimetrico} and \eqref{eqrhoepsappinL}.  Thus,
\begin{align*}
[v\overline{\rho_{\epsilon}}(j) + j\rho_{\epsilon}(v)](x,x) &= kv(x,x)j(x,x) + j(x,x)v(x,x) \\&= \begin{cases}
2v(x,x)j(x,x) &\text{if } k=1 \\
0 &\text{if } k=-1\\
\end{cases}
\\&= i(x,x).
\end{align*}

Conversely, if $\theta = \gamma (\widetilde{\rho_{\epsilon}}\circ \widetilde{k\delta})(\gamma)$  for some $\gamma = \begin{bmatrix}g\\j \end{bmatrix} \in U(D(X,K))$, then $f = g\rho_{\epsilon}(g)$ with $g \in U(FI(X,K))$. By \cite[Lemma~5.3]{BFS14}, $f(x,x) \in (K^*)^2$ for all $x \in X_3$.
\end{proof}

For each $f \in U(FI(X,K))$ $\rho_\lambda$-symmetric we define a map $\epsilon_f : X_3 \to K^*$ by $\epsilon_f(x)= f(x,x)$.

\begin{lemma}\label{lemarhoepsF}
Let $\theta = \begin{bmatrix}f\\i \end{bmatrix} \in U(D(X,K))$ and let $k \in \{-1,1\}$. If $\theta$ is ($\widetilde{\rho_\lambda} \circ \widetilde{k\delta}$)-symmetric, then  $\Phi = \Psi_\theta \circ \widetilde{\rho_\lambda} \circ \widetilde{k\delta}$ is an involution on $D(X,K)$  equivalent to $\widetilde{\rho_{\epsilon_f}} \circ \widetilde{k\delta}$, via inner automorphism.
\end{lemma}
\begin{proof}
Suppose $\theta$ is ($\widetilde{\rho_\lambda} \circ \widetilde{k\delta}$)-symmetric. Then $\Phi= \Psi_\theta \circ \widetilde{\rho_\lambda} \circ \widetilde{k\delta}$ is an involution on $D(X,K)$, by Proposition~\ref{propAutInner} (iii). By \eqref{eqrhoepsthetaeps2}, $\widetilde{\rho_{\epsilon_f}} = \Psi_{\theta_{\epsilon_f}}  \circ \widetilde{\rho_\lambda}$, where  $\theta_{\epsilon_f} = \begin{bmatrix} u_{\epsilon_f}\\ 0 \end{bmatrix}$. Thus
\begin{equation}\label{eqlemarhoepsF}
\Phi = \Psi_\theta \circ \widetilde{\rho_\lambda} \circ \widetilde{k\delta}  =  \Psi_\theta \circ \Psi_{\theta_{\epsilon_f}^{-1}}\circ \widetilde{\rho_{\epsilon_f}} \circ \widetilde{k\delta} =  \Psi_{\theta\theta_{\epsilon_f}^{-1}}\circ (\widetilde{\rho_{\epsilon_f}} \circ \widetilde{k\delta}).
\end{equation}
It follows from \eqref{eqrhoepsthetaeps1} and hypothesis that
\begin{align*}
(\widetilde{\rho_{\epsilon_f}} \circ \widetilde{k\delta})(\theta\theta_{\epsilon_f}^{-1}) &=
(\widetilde{\rho_{\epsilon_f}} \circ \widetilde{k\delta})(\theta_{\epsilon_f}^{-1}) (\widetilde{\rho_{\epsilon_f}} \circ \widetilde{k\delta})(\theta)
= \theta_{\epsilon_f}^{-1} (\Psi_{\theta_{\epsilon_f}} \circ \widetilde{\rho_\lambda} \circ \widetilde{k\delta})(\theta)
\\&=\theta_{\epsilon_f}^{-1}\theta_{\epsilon_f}(\widetilde{\rho_\lambda} \circ \widetilde{k\delta})(\theta)\theta_{\epsilon_f}^{-1}
=\theta\theta_{\epsilon_f}^{-1}.
\end{align*}
Moreover, for all $x \in X_3$,
\begin{align*}
(fu_{\epsilon_{f}}^{-1})(x,x) &= f(x,x)u_{\epsilon_{f}}(x,x)^{-1}= f(x,x)\epsilon_{f}(x)^{-1}\\&= f(x,x)f(x,x)^{-1} =1 \in (K^*)^2.
\end{align*}
By Lemma~\ref{lemaiffsimetricoo}, there is $\gamma \in U(D(X,K))$ such that $\theta\theta_{\epsilon_f}^{-1} = \gamma(\widetilde{\rho_{\epsilon_f}} \circ \widetilde{k\delta})(\gamma)$. Thus, by \eqref{eqlemarhoepsF} and  Proposition~\ref{propAutInner} (iv),  $\Phi$ and $\widetilde{\rho_{\epsilon_f}}\circ \widetilde{k\delta}$ are equivalent via inner automorphism.
\end{proof}

In order to classify  the involutions on $D(X,K)$ that induce $\lambda$ on $X$ via inner automorphisms, we consider two cases: when the poset $X$ has no fixed points with respect to $\lambda$ ($X_3= \emptyset$) and when $X$ has at least one fixed point ($X_3 \neq \emptyset)$.

\subsubsection{The case $X_3 = \emptyset$} \label{SEmpty}

In this subsection we suppose  $\lambda(x)\neq x$ for all $x \in X$. Consider $w \in FI(X,K)$ given by
$$w(x,y) = \begin{cases}
0 &\text{if } x\neq y\\
1 &\text{if } x = y \in X_1\\
-1 &\text{if } x = y \in X_2\\
\end{cases}.$$
Note that $w \in U(FI(X,K))$ with $w^{-1}= w$, and $\rho_\lambda(w)=-w$. Thus, $\sigma_\lambda = \Psi_w \circ \rho_\lambda$ is an involution on $FI(X,K)$, by Proposition~\ref{propAutInner} (iii). Let $\omega = \begin{bmatrix}w\\0 \end{bmatrix} \in U(D(X,K))$. By Remark~\ref{obsInnPsiepsitiu}, $\Psi_\omega = \widetilde{\Psi_w}$ and so $\widetilde{\sigma_\lambda}= \Psi_\omega\circ \widetilde{\rho_\lambda}$, by Remark~\ref{obsCompoExten}. Therefore, $\widetilde{\sigma_\lambda}$ and $\widetilde{\sigma_\lambda} \circ \widetilde{-\delta}$ are involutions on $D(X,K)$ that induce $\lambda$ on $X$, by  Corollary~\ref{coroSinal}. Also note that $\omega^{-1}=\omega$ and $\widetilde{\sigma_\lambda}(\omega) = - \omega$.

For $l \in I(X,K)$, we have
\begin{equation} \label{eqsiglamappinF}
    \overline{\sigma_\lambda}(l)(x,y) = \begin{cases}
-l(\lambda(y),\lambda(x)) &\text{if } x< y \text{ and } x \in X_1, y \in X_2 \\
l(\lambda(y),\lambda(x)) &\text{otherwise} \\
\end{cases}.
\end{equation}

\begin{lemma} \label{lemsigmalambdasimetrico}
Let $\Phi$ be an involution on $D(X,K)$ such that $\Phi = \Psi_\theta \circ \widetilde{\rho_\lambda} \circ \widetilde{s(\Phi)\delta}$ and $(\widetilde{\rho_\lambda} \circ \widetilde{s(\Phi)\delta})(\theta) = -\theta$. Then there exists  $\theta' \in U(D(X,K))$ such that $\Phi =  \Psi_{\theta'} \circ \widetilde{\sigma_\lambda} \circ \widetilde{s(\Phi)\delta}$ and $(\widetilde{\sigma_\lambda} \circ \widetilde{s(\Phi)\delta})(\theta')=\theta'$.
\end{lemma}
\begin{proof}
Since $\widetilde{\sigma_\lambda} = \Psi_{\omega} \circ \widetilde{\rho_\lambda}$ and $\omega^{-1}=\omega$, we have $\Phi =  \Psi_{\theta} \circ \Psi_\omega \circ  \widetilde{\sigma_\lambda} \circ \widetilde{s(\Phi)\delta}=  \Psi_{\theta\omega} \circ  \widetilde{\sigma_\lambda} \circ \widetilde{s(\Phi)\delta}$. Moreover,
\begin{align*}
     (\widetilde{\sigma_\lambda} \circ \widetilde{s(\Phi)\delta})(\theta\omega)&= (\widetilde{\sigma_\lambda} \circ \widetilde{s(\Phi)\delta})(\omega)(\widetilde{\sigma_\lambda} \circ \widetilde{s(\Phi)\delta})(\theta)
     =  \widetilde{\sigma_\lambda}(\omega)(\Psi_\omega \circ  \widetilde{\rho_\lambda} \circ \widetilde{s(\Phi)\delta})(\theta)
     \\&=-\omega\Psi_\omega(-\theta) =-\omega\omega(-\theta)\omega =\theta\omega.
\end{align*}
\end{proof}

\begin{lemma} \label{lemasigmasimetrico}
Let $\theta \in U(D(X,K))$ and $k \in \{-1,1\}$. If $\theta$ is  $(\widetilde{\sigma_\lambda}\circ \widetilde{k\delta})$-symmetric, then there exists $\gamma \in U(D(X,K))$ such that  $\theta = \gamma (\widetilde{\sigma_\lambda}\circ \widetilde{k\delta})(\gamma)$.
\end{lemma}
\begin{proof}
 If $\theta=\begin{bmatrix}f\\i \end{bmatrix}$ is  $(\widetilde{\sigma_\lambda}\circ \widetilde{k\delta})$-symmetric, then \begin{equation}\label{eqsigmassimetrico}
 \sigma_\lambda(f) = f\text{ and  }k\overline{\sigma_\lambda}(i) = i.
 \end{equation}
Let $v \in U(FI(X,K))$ and $j\in I(X,K)$ as defined in Lemma~\ref{lemaiffsimetricoo}. As in the proof of \cite[Lemma~16]{BFS11}, $f = v\sigma_\lambda(v)$. Consider $\gamma = \begin{bmatrix}v\\j \end{bmatrix} \in U(D(X,K))$. We want to prove that $\theta = \gamma (\widetilde{\sigma_\lambda}\circ \widetilde{k\delta})(\gamma)$. For that, it remains to prove that
\begin{equation}\label{eqsigmasimetrico}
 i = kv\overline{\sigma_\lambda}(j) + j\sigma_{\lambda}(v).
\end{equation}

Let $x \leq y$ in $X$. Since $X_3= \emptyset$, we have to consider only the first three cases in the proof of Lemma~\ref{lemaiffsimetricoo}. For the first two cases the proof of \eqref{eqsigmasimetrico} is exactly the same as the proof of \eqref{eqigualdades}, since $\overline{\sigma_\lambda}(j)(x,y) = \overline{\rho_\lambda}(j)(x,y)$ and  $\sigma_\lambda(v)(x,y)=\rho_\lambda(v)(x,y)$, by \eqref{eqsiglamappinF}. Finally, suppose $x \in X_1$ and $y \in X_2$ and let $t \in X$ such that $x < t < y$. In this case,  $t \in X_1 \cup X_2$. If $t \in X_1$, then $v(x,t) = 0$ and $j(x,t)=0$.  On the other hand, if $t \in X_2$, then $\overline{\sigma_\lambda}(j)(t,y)= j(\lambda(y),\lambda(t))=0$ and  $\sigma_\lambda(v)(t,y) = v(\lambda(y),\lambda(t))=0$, by \eqref{eqsiglamappinF}. Thus, by \eqref{eqsiglamappinF} and \eqref{eqsigmassimetrico},
\begin{align} \label{eqsigmasimetricoP1}
 kv\overline{\sigma_\lambda}(j)(x,y) &= k[v(x,x)\overline{\sigma_\lambda}(j)(x,y) + v(x,y)\overline{\sigma_\lambda}(j)(y,y)] \nonumber\\
  &= k[\overline{\sigma_\lambda}(j)(x,y) + v(x,y)j(\lambda(y),\lambda(y))] \nonumber\\
  &= k\overline{\sigma_\lambda}(j)(x,y) = - kj(\lambda(y),\lambda(x)) \nonumber\\
  &= \dfrac{ -ki(\lambda(y),\lambda(x))}{2} =\dfrac {k\overline{\sigma_\lambda}(i)(x,y)}{2} \nonumber\\
  &= \dfrac{i(x,y)}{2}
\end{align}
and
 \begin{align} \label{eqsigmasimetricoP2}
  j\sigma_\lambda(v)(x,y)&= j(x,x)\sigma_\lambda(v)(x,y) + j(x,y)\sigma_\lambda(v)(y,y)\nonumber\\&=  j(x,y)v(\lambda(y),\lambda(y)) = j(x,y) = \dfrac{i(x,y)}{2}.
 \end{align}
 Therefore, $[kv\overline{\sigma_\lambda}(j) + j\sigma_\lambda(v)](x,y) =i(x,y)$, by \eqref{eqsigmasimetricoP1} and \eqref{eqsigmasimetricoP2}.
\end{proof}

\begin{theorem} \label{teoclassivazio}
Let $\lambda$ be an involution on $X$ such that $\lambda(x) \neq x$, for all $x \in X$. If $\Phi$ is an involution on $D(X,K)$ that
induces $\lambda$, then there exists $\Psi \in \Inn(D(X,K))$ such that $\Phi \circ \Psi = \Psi \circ \widetilde{\rho_\lambda}\circ \widetilde{s(\Phi)\delta}$ or $\Phi \circ \Psi = \Psi \circ \widetilde{\sigma_\lambda}\circ \widetilde{s(\Phi)\delta}$.
\end{theorem}
\begin{proof}
By Corollaries~\ref{coroallcomparableAntiauttDP} and \ref{coroSinal} and  Proposition~\ref{propmaisoummenosTheta}, there exists $\theta \in U(D(X,K))$ such that $\Phi= \Psi_\theta \circ \widetilde{\rho_\lambda} \circ \widetilde{s(\Phi)\delta}$ and $(\widetilde{\rho_\lambda} \circ \widetilde{s(\Phi)\delta})(\theta) = \pm \theta$. If $(\widetilde{\rho_\lambda} \circ \widetilde{s(\Phi)\delta})(\theta) = \theta$, then  there exists $\Psi \in \Inn(D(X,K))$ such that $\Phi \circ \Psi = \Psi \circ \widetilde{\rho_\lambda}\circ \widetilde{s(\Phi)\delta}$, by Lemma~\ref{lemarhoepsF}, since $X_3 \neq \emptyset$.

If $(\widetilde{\rho_\lambda} \circ \widetilde{s(\Phi)\delta})(\theta) = -\theta$, then by Lemma~\ref{lemsigmalambdasimetrico}, there exists $\theta' \in U(D(X,K))$ such that $\Phi= \Psi_{\theta'} \circ \widetilde{\sigma_\lambda}\circ \widetilde{s(\Phi)\delta}$ and $(\widetilde{\sigma_\lambda}\circ \widetilde{s(\Phi)\delta})(\theta')= \theta'$. So, by Lemma~\ref{lemasigmasimetrico}, there is $\gamma \in U(D(X,K))$ such that $\theta' = \gamma(\widetilde{ \sigma_\lambda}\circ \widetilde{s(\Phi)\delta})(\gamma)$. Therefore, $\Phi \circ \Psi = \Psi \circ \widetilde{\sigma_\lambda}\circ \widetilde{s(\Phi)\delta}$ for some $\Psi \in \Inn(D(X,K))$, by Proposition~\ref{propAutInner} (iv).
\end{proof}
As stated in \cite[p.1097]{BFS14}, the results from \cite{BFS11} about the classification of involutions via inner automorphisms remain valid for $FI(X,K)$ for any connected poset $X$ for which every multiplicative automorphism of $FI(X,K)$ is inner. In particular, the involutions $\rho_\lambda$ and $\sigma_\lambda$ are not equivalent via inner automorphisms. Therefore, by Corollary~\ref{coroinduzesinal} and Proposition~\ref{propequivalenciainner}, the involutions $\widetilde{\rho_\lambda}, \widetilde{\rho_\lambda} \circ \widetilde{-\delta},\widetilde{\sigma_\lambda}$ and $\widetilde{\sigma_\lambda} \circ \widetilde{-\delta}$ on $D(X,K)$ are pairwise not equivalent via inner automorphisms.

\subsubsection{The case $X_3 \neq \emptyset$} \label{SnoEmpty}

In this subsection, $\lambda$ has at least one fixed point in $X$.

\begin{lemma} \label{lemanaoexistemenostheta}
Let $k  \in \{-1,1\}$. There is no $\theta \in U(D(X,K))$ such that $(\widetilde{\rho_\lambda}\circ \widetilde{k\delta})(\theta) =-\theta$.
\end{lemma}
\begin{proof}
If $\theta = \begin{bmatrix} f\\i \end{bmatrix} \in D(X,K)$ is such that $(\widetilde{\rho_\lambda}\circ \widetilde{k\delta})(\theta) =-\theta$, then $\rho_\lambda(f) = -f$. Let $x \in X_3$. Then $-f(x,x) = \rho_\lambda(f)(x,x)= f(\lambda(x),\lambda(x))=f(x,x)$, which implies $f(x,x)= 0$, since $\Char K \neq 2$. Thus, $f\notin U(FI(X,K))$, therefore $\theta \notin U(D(X,K))$.
\end{proof}

\begin{lemma} \label{lemaequiRhoEps}
Let $\lambda$ be an involution on $X$ with at least one fixed point in $X$. Let $\Phi$ be an involution on $D(X,K)$ that induces $\lambda$ on $X$. Then $\Phi$ is equivalent, via inner automorphism, to $\widetilde{\rho_{\epsilon_f}} \circ \widetilde{s(\Phi)\delta}$ for some $f \in U(FI(X,K))$.
\end{lemma}
\begin{proof}
By Corolaries~\ref{coroallcomparableAntiauttDP} and \ref{coroSinal}, there exists $\theta=  \begin{bmatrix}f\\i \end{bmatrix}  \in U(D(X,K))$ such that $\Phi = \Psi_\theta \circ \widetilde{\rho_\lambda} \circ \widetilde{s(\Phi)\delta}$. Moreover, by Proposition~\ref{propmaisoummenosTheta} and Lemma~\ref{lemanaoexistemenostheta}, we can choose $\theta$ such that $(\widetilde{\rho_\lambda} \circ \widetilde{s(\Phi)\delta})(\theta) = \theta$. Therefore, by Lemma~\ref{lemarhoepsF}, $\Phi$ is equivalent to $\widetilde{\rho_{\epsilon_f}} \circ \widetilde{s(\Phi)\delta}$, via inner automorphism.
\end{proof}

It follows from lemma above that every involution $\Phi$ on $D(X,K)$ that induces $\lambda$ on $X$ is equivalent, via inner automorphism, to $\widetilde{\rho_\epsilon} \circ \widetilde{s(\Phi)\delta}$ for some map $\epsilon: X_3 \to K^*$. Therefore, to get a classification of involutions, via inner automorphisms, it suffices to discover which conditions the maps $\epsilon_1,\epsilon_2: X_3 \to K^*$ must satisfy in order that to obtain $\rho_{\epsilon_1}$ and $\rho_{\epsilon_2}$ equivalent via inner automorphism, by Corollary~\ref{coroinduzesinal} and Proposition~\ref{propequivalenciainner}.

As in \cite{BFS11}, we denote by $S_K$ the multiplicative group $K/(K^*)^2$. For each $\epsilon:X_3 \to K^*$, we define $\chi_\epsilon: X_3 \to S_K$ by $\chi_\epsilon = \pi \circ \epsilon$, where $\pi : K^* \to K/(K^*)^2$ is the canonical homomorphism.

\begin{lemma} \label{lemaEquivalenciarhoeps}
Let $\epsilon_1,\epsilon_2:X_3 \to K^*$be maps. The involutions $\rho_{\epsilon_1}$ and $\rho_{\epsilon_2}$ on $FI(X,K)$ are equivalent, via inner automorphism, if and only if there exists $g \in S_K$ such that $g \chi_{\epsilon_1}(x) = \chi_{\epsilon_2}(x)$, for all $x \in X_3$.
\end{lemma}
\begin{proof}
The proof is exactly the same as the proof of \cite[Lemma~22]{BFS11}, just replacing \cite[Lemma~11]{BFS11} by \cite[Lemma~5.3]{BFS14}.
\end{proof}

Let $k \in \{-1,1\}$. By Proposition~\ref{propequivalenciainner} and Lemma~\ref{lemaEquivalenciarhoeps}, $\widetilde{\rho_{\epsilon_1}} \circ \widetilde{k\delta}$ is equivalent to $\widetilde{\rho_{\epsilon_2}} \circ \widetilde{k\delta}$ via inner automorphism, if and only if $L_g \circ \chi_{\epsilon_1} = \chi_{\epsilon_2}$ for some $g \in S_K$, where $L_g:S_K \to S_K$ is the left multiplication by $g$. It is easy to see that, for any $\chi: X_3 \to S_K$, there exists $\epsilon:X_3\to K^*$ such that $\chi = \chi_\epsilon$. Let $S_K^{X_3}$ the set of all maps from $X_3$ to $S_K$. Then,  by Lemma~\ref{lemaequiRhoEps}, the number of equivalence classes of involutions on $D(X,K)$ inducing $\lambda$, via inner automorphism, is equal to the number of equivalence classes on the set $S_K^{X_3} \times \{-1,1\}$ with the following equivalence relation: $(\chi, x),(\chi', y) \in S_K^{X_3} \times \{-1,1\}$ are equivalent if $x=y$ and there exists $g \in S_K$ such that $\chi = L_g \circ \chi'$.

By Lemma~\ref{lemaequiRhoEps}, if $\Phi$ is an involution on $D(X,K)$ that induces $\lambda$ on $X$, then there exists $\epsilon: X_3 \to K^*$ such that $\Phi$ is equivalent to $\widetilde{\rho_\epsilon} \circ \widetilde{s(\Phi)\delta}$, via inner automorphism. Thus, $\Phi$ induces the equivalence class in $S_K^{X_3} \times \{-1,1\}$ that contains $(\chi_\epsilon,s(\Phi))$.  Therefore we have the following theorem.

\begin{theorem} \label{teoclassinoempty}
Let $\lambda$ be an involution on $X$ with at least one fixed point in $X$. Let $\Phi_1$ and $\Phi_2$ be involutions on $D(X,K)$ that induce $\lambda$ in $X$. Then $\Phi_1$ and $\Phi_2$ are equivalent via inner automorphism, if and only if they induce the same equivalence class in $S_K^{X_3} \times \{-1,1\}$.
\end{theorem}

The cardinality of $S_K^{X_3} \times \{-1,1\} $ is $2|S_K|^{|X_3|}$, but in each equivalence class there are $|S_K|$ elements (one for each possible left multiplication by an element of $S_K$), then we have the following corollary.

\begin{corollary} \label{corocardinality}
Let $\lambda$ be an involution on $X$ with at least one fixed point in $X$. The number of equivalence classes of involutions on $D(X,K)$ inducing $\lambda$, via inner automorphism, is equal to $2|S_K|^{|X_3|-1}$.
\end{corollary}

\subsection{General classification of involutions on $D(X,K)$}\label{Sclassigeneral}

Recall we are considering a field $K$ of characteristic different from $2$ and $X$ a connected poset such that $\Mult(FI(X,K)) \subseteq \Inn(FI(X,K))$ and $\Der(X,K) = \IDer(X,K)$.

The general classification of involutions on $D(X,K)$ can be obtained naturally from the classification via inner automorphism, by the following result.

\begin{theorem}~\label{teoeqivaGeral}
The involutions $\Phi_1$ and $\Phi_2$ on $D(X,K)$ are equivalent if and only if there exist an automorphism $\alpha$ of $X$ and $k \in K^*$ such that the involutions $\Phi_1$ and $\widetilde{\alpha} \circ \widetilde{k\delta} \circ \Phi_2 \circ \widetilde{k^{-1}\delta}\circ \widetilde{\alpha^{-1}}$ are equivalent, via inner automorphism.
\end{theorem}
\begin{proof}
The involutions $\Phi_1$ and $\Phi_2$ on $D(X,K)$ are equivalent if and only if there exists an automorphism $\varphi$ of $D(X,K)$ such that $\Phi_1 = \varphi \circ \Phi_2 \circ \varphi^{-1}$. By Proposition~\ref{propallcomparableAutDP}, $\varphi= \Psi \circ \widetilde{\alpha}\circ \widetilde{k\delta}$, where $\Psi \in \Inn(D(X,K))$, $\alpha$ is an automorphism of $X$ and $k \in K^*$. Thus, $\Phi_1$ and $\Phi_2$ are equivalent if and only if $\Phi_1 = \Psi \circ \widetilde{\alpha} \circ \widetilde{k\delta} \circ \Phi_2 \circ \widetilde{k^{-1}\delta}\circ \widetilde{\alpha^{-1}}\circ \Psi^{-1}$.
\end{proof}

\begin{remark} \label{obsalplambalp}
Let $\alpha$ and $\lambda$ be an automorphism and an involution on $X$, respectively. Then $\alpha\circ\lambda\circ \alpha^{-1}$ is an involution on $X$ and $\widehat{\alpha} \circ \rho_\lambda \circ \widehat{\alpha^{-1}} = \rho_{\alpha\circ\lambda\circ\alpha^{-1}}$. Therefore, $\widetilde{\alpha}\circ \widetilde{\rho_\lambda} \circ \widetilde{\alpha^{-1}} =\widetilde{\rho_{\alpha\circ\lambda\circ\alpha^{-1}}}$, by Remark~\ref{obsCompoExten}.
\end{remark}

\begin{proposition} \label{propConjuInv}
Let $\Phi$ be an involution on  $D(X,K)$ that induces $\lambda$ on $X$, $\alpha$ an automorphism of $X$ and $k \in K^*$. Then the involution $\Phi_1 = \widetilde{\alpha} \circ \widetilde{k\delta} \circ \Phi \circ \widetilde{k^{-1}\delta}\circ \widetilde{\alpha^{-1}}$ induces the involution $\alpha\circ \lambda\circ \alpha^{-1}$ on $X$ and $s(\Phi_1) =  s(\Phi)$.
\end{proposition}
\begin{proof}
By Corollary~\ref{coroallcomparableAntiauttDP} and Remark~\ref{obsInnPsiepsitiu}, $\Phi = \widetilde{\Psi_{f}} \circ \widetilde{D_j} \circ \widetilde{\rho_{\lambda}} \circ \widetilde{ s(\Phi)\delta}$ for some $f \in U(FI(X,K))$ and $j \in I(X,K)$. Thus, by \eqref{comutaautparticular}, \eqref{comutaaut}, \eqref{comutaantiautparticular} and Remark~\ref{obsalplambalp},
\begin{align*} \label{eqinducesconjulambd}
 \Phi_1 &= \widetilde{\alpha} \circ \widetilde{k\delta} \circ \widetilde{\Psi_{f}} \circ \widetilde{D_j} \circ \widetilde{\rho_{\lambda}} \circ \widetilde{ s(\Phi)\delta} \circ \widetilde{k^{-1}\delta}\circ \widetilde{\alpha^{-1}}\nonumber\\
 &= \widetilde{\alpha} \circ \widetilde{\Psi_{f}} \circ \widetilde{k\delta}  \circ \widetilde{D_j} \circ \widetilde{\rho_{\lambda}} \circ \widetilde{k^{-1}\delta}\circ \widetilde{\alpha^{-1}}\circ \widetilde{ s(\Phi)\delta} \nonumber\\
 &= \widetilde{\alpha}  \circ \widetilde{\Psi_{f}} \circ \widetilde{k D_j}\circ \widetilde{k\delta} \circ \widetilde{\rho_{\lambda}} \circ \widetilde{k^{-1}\delta}\circ \widetilde{\alpha^{-1}}\circ \widetilde{ s(\Phi)\delta}\nonumber\\
 &=  \widetilde{\alpha}  \circ \widetilde{\Psi_{f}} \circ \widetilde{ D_{kj}} \circ \widetilde{\rho_{\lambda}} \circ \widetilde{\alpha^{-1}}\circ \widetilde{ s(\Phi)\delta}\nonumber\\
 &=  (\widetilde{\alpha}  \circ \widetilde{\Psi_{f}} \circ \widetilde{ D_{kj}}\circ \widetilde{\alpha^{-1}}) \circ (\widetilde{\alpha} \circ \widetilde{\rho_{\lambda}} \circ \widetilde{\alpha^{-1}})\circ \widetilde{ s(\Phi)\delta}\nonumber \\
 &= (\widetilde{\alpha}  \circ \widetilde{\Psi_{f}} \circ \widetilde{ D_{kj}}\circ \widetilde{\alpha^{-1}}) \circ \widetilde{\rho_{\alpha\circ \lambda \circ\alpha^{-1}}}\circ \widetilde{ s(\Phi)\delta}.
 \end{align*}
By \cite[Proposition~7]{DW}, $\widetilde{\Psi_{f}} \circ \widetilde{ D_{kj}} \in \Inn(D(X,K))$. Since $\Inn(D(X,K))$ is a normal subgroup of $\Aut(D(X,K))$,  there exists $\theta' \in U(D(X,K))$ such that $\Psi_{\theta'} = \widetilde{\alpha}  \circ \widetilde{\Psi_{f}} \circ \widetilde{ D_{kj}}\circ \widetilde{\alpha^{-1}}$. Therefore, $\Phi_1 = \Psi_{\theta'} \circ   \widetilde{\rho_{\alpha\circ\lambda \circ \alpha^{-1}}} \circ \widetilde{ s(\Phi)\delta}$ and so $\lambda_{\Phi_1} = \alpha\circ\lambda \circ \alpha^{-1}$ and $ s(\Phi_1) =  s(\Phi)$, by Corollary~\ref{coroSinal}.
\end{proof}

In \cite[p.1953]{BFS11}, the authors define an equivalence relation $\sim$ on the set of all involutions on $X$ as follows: $\lambda \sim \mu$ if there exists an automorphism $\alpha$ of $X$ such that $\alpha\circ \lambda = \mu \circ \alpha$.

\begin{corollary} \label{coroInduzesinalGeral}
\begin{enumerate}
\item[(i)] If the involutions $\Phi_1$ and $\Phi_2$ on $D(X,K)$ are equivalent, then  $s(\Phi_1)=s(\Phi_2)$ and  $\lambda_{\Phi_1} \sim \lambda_{\Phi_2}$.
\item[(ii)] If the involutions $\lambda_1$ and $\lambda_2$ on $X$ are equivalent and $\Phi_1$ is an involution on $D(X,K)$ that induces $\lambda_1$, then $\Phi_1$ is equivalent to some involution $\Phi_2$ on $D(X,K)$ that induces $\lambda_2$.
\end{enumerate}
\end{corollary}
\begin{proof}
(i) Suppose the involutions  $\Phi_1$ and $\Phi_2$ on $D(X,K)$ are equivalent. By Theorem~\ref{teoeqivaGeral}, there exist an automorphism $\alpha$ of $X$ and $k\in K^*$ such that $\Phi_1$ is equivalent to $\Phi_3= \widetilde{\alpha} \circ \widetilde{k\delta} \circ \Phi_2 \circ \widetilde{k^{-1}\delta}\circ \widetilde{\alpha^{-1}}$, via inner automorphism. By Corollary~\ref{coroinduzesinal}, $\lambda_{\Phi_1}= \lambda_{\Phi_3}$ and $s(\Phi_1)= s(\Phi_3)$. On the other hand, $\lambda_{\Phi_3}= \alpha\circ \lambda_{\Phi_2}\circ \alpha^{-1}$ and $s(\Phi_3)= s(\Phi_2)$, by Proposition~\ref{propConjuInv}. Therefore, $\lambda_{\Phi_1} = \alpha\circ \lambda_{\Phi_2}\circ\alpha^{-1}$ and $s(\Phi_1)= s(\Phi_2)$.

(ii) Let $\alpha$ be an automorphism of $X$ such that $\lambda_2 = \alpha \circ \lambda_1 \circ \alpha^{-1}$ and consider the involution $\Phi_2 = \widetilde{\alpha} \circ \Phi_1 \circ \widetilde{\alpha^{-1}}$ on $D(X,K)$. Then $\Phi_1$ and $\Phi_2$ are equivalent and, by Proposition~\ref{propConjuInv}, $\Phi_2$ induces $\alpha \circ \lambda_1 \circ \alpha^{-1} = \lambda_2$.
\end{proof}

The following three theorems tell us how to proceed to determine the equivalence class of an involution on $D(X,K)$.

\begin{theorem} \label{teoclassigeralnoFP}
Let $\lambda$ be an involution on $X$ for which there is no fixed points in $X$. Then the involutions $\widetilde{\rho_\lambda}, \widetilde{\rho_\lambda} \circ \widetilde{-\delta}, \widetilde{\sigma_\lambda}$ and $\widetilde{\sigma_\lambda} \circ \widetilde{-\delta}$ are pairwise not equivalent. Furthermore, each involution on $D(X,K)$ that induces $\lambda$ on $X$ is equivalent to one of them.
\end{theorem}
\begin{proof}
By Theorem~\ref{teoeqivaGeral}, $\widetilde{\rho_\lambda}$ and $\widetilde{\sigma_\lambda}$ are equivalent if and only if there exist an automorphism $\alpha$ of $X$ and $k \in K^*$ such that $\widetilde{\sigma_\lambda}$ and $\Phi= \widetilde{\alpha}\circ \widetilde{k\delta} \circ \widetilde{\rho_\lambda} \circ \widetilde{k^{-1}\delta}\circ \widetilde{\alpha^{-1}}$ are equivalent via inner automorphism. By Corollary~\ref{coroinduzesinal} and Proposition~\ref{propConjuInv}, $\lambda = \alpha \circ \lambda \circ \alpha^{-1}$ and, by Remark~\ref{obsalplambalp}, $\widehat{\alpha} \circ \rho_\lambda \circ \widehat{\alpha^{-1}} = \rho_{\alpha\circ \lambda \circ \alpha^{-1}} = \rho_\lambda$, therefore $\widetilde{\alpha} \circ \widetilde{\rho_\lambda} \circ \widetilde{\alpha^{-1}} = \widetilde{\rho_\lambda}$. Note that $\Phi =  \widetilde{\alpha} \circ \widetilde{\rho_\lambda} \circ \widetilde{\alpha^{-1}} = \widetilde{\rho_\lambda}$, by \eqref{comutaantiautparticular}. Thus, $\widetilde{\sigma_\lambda}$ and $\widetilde{\rho_\lambda}$ are equivalent via inner automorphism, which is a contradiction. Therefore, the involutions $\widetilde{\rho_\lambda}, \widetilde{\rho_\lambda} \circ \widetilde{-\delta}, \widetilde{\sigma_\lambda}$ and $\widetilde{\sigma_\lambda} \circ \widetilde{-\delta}$ are pairwise not equivalent, by Proposition~\ref{propequivalenciainner} and Corollary~\ref{coroInduzesinalGeral} (i).

Finally, each involution on $D(X,K)$ that induces $\lambda$ on $X$ is equivalent (via inner automorphism) to one of the involutions  $\widetilde{\rho_\lambda}, \widetilde{\rho_\lambda} \circ \widetilde{-\delta}, \widetilde{\sigma_\lambda}$ and $\widetilde{\sigma_\lambda} \circ \widetilde{-\delta}$, by Theorem~\ref{teoclassivazio}.
\end{proof}

\begin{theorem} \label{teoclassigeralonepoint}
Let $\lambda$ be an involution on $X$ with exactly one fixed point in $X$. Each involution on $D(X,K)$ that induces $\lambda$ on $X$ is equivalent to either $\widetilde{\rho_\lambda}$ or $\widetilde{\rho_\lambda} \circ \widetilde{-\delta}$.
\end{theorem}
\begin{proof}
The involutions $\widetilde{\rho_\lambda}$ and $\widetilde{\rho_\lambda} \circ \widetilde{-\delta}$ are not equivalent, by Corollary~\ref{coroInduzesinalGeral} (i). By Corollary~\ref{corocardinality}, the number of equivalence classes of involutions on $D(X,K)$ inducing $\lambda$, via inner automorphism, is $2|S_K|^{1-1}=2$. Therefore, each involution on $D(X,K)$ that induces $\lambda$ on $X$ is equivalent (via inner automorphism) to either $\widetilde{\rho_\lambda}$ or $\widetilde{\rho_\lambda} \circ \widetilde{-\delta}$.
\end{proof}

\begin{remark}
Let $\lambda$ be an involution on $X$ and $\alpha$ an automorphism of $X$ such that $\alpha \circ \lambda = \lambda \circ \alpha$. Let $(X_1,X_2, X_3)$ be a $\lambda$-decomposition of $X$. If $x \in X_3$, then $\alpha(x) \in X_3$ (see \cite[Lemma~29]{BFS11}). Thus, if $\epsilon: X_3 \to S_K$ is a map, then $\epsilon \circ \alpha :X_3 \to S_K$ is well defined.
\end{remark}

\begin{theorem} \label{teoclassigeralfinal}
Let $\lambda$ be an involution on $X$ for which there is more than one fixed point in $X$. Let $\epsilon_1,\epsilon_2: X_3 \to K^*$ be maps such that $\rho_{\epsilon_1}$ and  $\rho_{\epsilon_2}$ are not equivalent via inner automorphism. Then $\rho_{\epsilon_1}$ and $\rho_{\epsilon_2}$ are equivalent if and only if there exists an automorphism $\alpha:X \to X$ such that $\alpha \circ \lambda = \lambda\circ \alpha$, and $\rho_{\epsilon_2\circ \alpha}$ and $\rho_{\epsilon_1}$ are equivalent, via inner automorphism.
\end{theorem}
\begin{proof}
Firstly, note that if $\rho$ is an involution on $FI(X,K)$ that induces $\lambda$ on $X$ and $\alpha$ is an automorphism of $X$, then $\widetilde{\alpha}\circ \widetilde{\rho} \circ \widetilde{\alpha^{-1}}$ induces $\alpha\circ \lambda\circ\alpha^{-1}$ on $X$, by Proposition~\ref{propConjuInv}. Since $\widetilde{\alpha}\circ \widetilde{\rho} \circ \widetilde{\alpha^{-1}} = \begin{bmatrix}
\widehat{\alpha}\circ \rho \circ \widehat{\alpha^{-1}}& 0\\ 0 & \overline{\alpha}\circ \overline{\rho} \circ \overline{\alpha^{-1}}&
\end{bmatrix}$, $\widehat{\alpha}\circ \rho \circ \widehat{\alpha^{-1}}$ induces $\alpha\circ \lambda\circ\alpha^{-1}$ on $X$, by Theorem~\ref{teoDecinvDP}.

Now, the proof is the same of \cite[Theorem~31]{BFS11} just replacing Corollary~6 and Theorem~25 from \cite{BFS11} by Theorem~5.1 and Theorem~5.6 from \cite{BFS14}, respectively, and \cite[Proposition~26]{BFS11} by the result in the paragraph above.
\end{proof}

\section*{Acknowledgments}
The second author was financed in part by the Coordena\c{c}\~{a}o de Aperfei\c{c}oamento de Pessoal de N\'{i}vel Superior - Brasil (CAPES) - Finance 001.

\end{document}